\documentclass[12pt, reqno]{amsart}
\usepackage{txfonts}      
\usepackage{amssymb}
\usepackage{eucal}
\usepackage{amsmath}
\usepackage{amsthm}
\usepackage{amscd}
\usepackage[dvips]{color}
\usepackage{multicol}
\usepackage[all]{xy}           
\usepackage{color}
\usepackage{colordvi}
\usepackage{xspace}
\usepackage{tikz}
\usepackage{ifpdf}
\ifpdf
\usepackage[colorlinks,final,backref=page,hyperindex]{hyperref}
\else
\usepackage[colorlinks,final,backref=page,hyperindex,hypertex]{hyperref}
\fi

\newcommand {\emptycomment}[1]{} 

\newcommand{\nc}{\newcommand}
\newcommand{\delete}[1]{}

\nc{\mlabel}[1]{\label{#1}}  
\nc{\mcite}[1]{\cite{#1}}  
\nc{\mref}[1]{\ref{#1}}  
\nc{\meqref}[1]{\eqref{#1}} 
\nc{\mbibitem}[1]{\bibitem{#1}} 

\delete{
\nc{\mlabel}[1]{\label{#1}  
{\hfill \hspace{1cm}{\bf{{\ }\hfill(#1)}}}}
\nc{\mcite}[1]{\cite{#1}{{\bf{{\ }(#1)}}}}  
\nc{\mref}[1]{\ref{#1}{{\bf{{\ }(#1)}}}}  
\nc{\meqref}[1]{\eqref{#1}{{\bf{{\ }(#1)}}}} 
\nc{\mbibitem}[1]{\bibitem[\bf #1]{#1}} 
}

\newtheorem{thm}{Theorem}[section]
\newtheorem{lem}[thm]{Lemma}
\newtheorem{cor}[thm]{Corollary}
\newtheorem{pro}[thm]{Proposition}
\theoremstyle{definition}
\newtheorem{defi}[thm]{Definition}
\newtheorem{ex}[thm]{Example}
\newtheorem{rmk}[thm]{Remark}

\nc{\tred}[1]{\textcolor{red}{#1}}
\nc{\tblue}[1]{\textcolor{blue}{#1}}
\nc{\tgreen}[1]{\textcolor{green}{#1}}
\nc{\tpurple}[1]{\textcolor{purple}{#1}}
\nc{\btred}[1]{\textcolor{red}{\bf #1}}
\nc{\btblue}[1]{\textcolor{blue}{\bf #1}}
\nc{\btgreen}[1]{\textcolor{green}{\bf #1}}
\nc{\btpurple}[1]{\textcolor{purple}{\bf #1}}

\nc{\rk}{\mathrm{r}}

\nc{\tforall}{\text{ for all }}

\nc{\svec}[2]{{\tiny\left(\begin{matrix}#1\\
#2\end{matrix}\right)\,}}  
\nc{\ssvec}[2]{{\tiny\left(\begin{matrix}#1\\
#2\end{matrix}\right)\,}} 
\nc{\typeI}{local cocycle $3$-Lie bialgebra\xspace}
\nc{\typeIs}{local cocycle $3$-Lie bialgebras\xspace}
\nc{\typeII}{double construction $3$-Lie bialgebra\xspace}
\nc{\typeIIs}{double construction $3$-Lie bialgebras\xspace}
\nc{\bia}{{$\mathcal{P}$-bimodule ${\bf k}$-algebra}\xspace}
\nc{\bias}{{$\mathcal{P}$-bimodule ${\bf k}$-algebras}\xspace}
\nc{\rmi}{{\mathrm{I}}}
\nc{\rmii}{{\mathrm{II}}}
\nc{\rmiii}{{\mathrm{III}}}

\nc{\OT}{constant $\theta$-}
\nc{\T}{$\theta$-}
\nc{\IT}{inverse $\theta$-}

\nc{\asi}{ASI\xspace}
\nc{\adm}{admissible\xspace}
\nc{\qadm}{$Q$-admissible\xspace}
\nc{\aybe}{AYBE\xspace}
\nc{\aybes}{AYBEs\xspace}
\nc{\admset}{\{\pm x\}\cup (-x+K^{\times}) \cup K^{\times} x^{-1}}
\nc{\dualrep}{gives a dual representation\xspace}
\nc{\admt}{admissible to\xspace}
\nc{\opa}{\cdot_A}
\nc{\opb}{\cdot_B}

\nc{\dfl}{\succ} \nc{\dfr}{\prec}
\nc{\disp}[1]{\displaystyle{#1}}
\nc{\bin}[2]{ (_{\stackrel{\scs{#1}}{\scs{#2}}})}  
\nc{\binc}[2]{ \left (\!\! \begin{array}{c} \scs{#1}\\
\scs{#2} \end{array}\!\! \right )}  
\nc{\bincc}[2]{  \left ( {\scs{#1} \atop
\vspace{-.5cm}\scs{#2}} \right )}  
\nc{\sarray}[2]{\begin{array}{c}#1 \vspace{.1cm}\\ \hline
\vspace{-.35cm} \\ #2 \end{array}}
\nc{\bs}{\bar{S}} \nc{\dcup}{\stackrel{\bullet}{\cup}}
\nc{\dbigcup}{\stackrel{\bullet}{\bigcup}} \nc{\etree}{\big |}
\nc{\la}{\longrightarrow} \nc{\fe}{\'{e}} \nc{\rar}{\rightarrow}
\nc{\dar}{\downarrow}
\nc{\ot}{\otimes} \nc{\sot}{{\scriptstyle{\ot}}}
\nc{\otm}{\overline{\ot}}
\nc{\ora}[1]{\stackrel{#1}{\rar}}
\nc{\ola}[1]{\stackrel{#1}{\la}}
\nc{\pltree}{\calt^\pl}
\nc{\epltree}{\calt^{\pl,\NC}}
\nc{\rbpltree}{\calt^r}
\nc{\scs}[1]{\scriptstyle{#1}} \nc{\mrm}[1]{{\rm #1}}
\nc{\dirlim}{\displaystyle{\lim_{\longrightarrow}}\,}
\nc{\invlim}{\displaystyle{\lim_{\longleftarrow}}\,}
\nc{\sha}{{\mbox{\cyr X}}}  
\nc{\ncsha}{{\mbox{\cyr X}^{\mathrm NC}}} \nc{\ncshao}{{\mbox{\cyr
X}^{\mathrm NC,\,0}}}
\nc{\shpr}{\diamond}    
\nc{\shprm}{\overline{\diamond}}    
\nc{\shpro}{\diamond^0}    
\nc{\shprr}{\diamond^r}     
\nc{\shpra}{\overline{\diamond}^r}
\nc{\shpru}{\check{\diamond}} \nc{\catpr}{\diamond_l}
\nc{\rcatpr}{\diamond_r} \nc{\lapr}{\diamond_a}
\nc{\sqcupm}{\ot}
\nc{\lepr}{\diamond_e} \nc{\vep}{\varedhlon} \nc{\labs}{\mid\!}
\nc{\rabs}{\!\mid} \nc{\hsha}{\widehat{\sha}}
\nc{\lsha}{\stackrel{\leftarrow}{\sha}}
\nc{\rsha}{\stackrel{\rightarrow}{\sha}} \nc{\lc}{\lfloor}
\nc{\rc}{\rfloor}
\nc{\tpr}{\sqcup}
\nc{\nctpr}{\vee}
\nc{\plpr}{\ob}
\nc{\rbplpr}{\bar{\plpr}}
\nc{\sqmon}[1]{\langle #1\rangle}
\nc{\forest}{\calf}
\nc{\altx}{\Lambda_X} \nc{\vecT}{\vec{T}} \nc{\onetree}{\bullet}
\nc{\Ao}{\check{A}}
\nc{\seta}{\underline{\Ao}}
\nc{\deltaa}{\overline{\delta}}
\nc{\trho}{\tilde{\rho}}

\nc{\rpr}{\rhd}
\nc{\dpr}{{\tiny\diamond}}
\nc{\rprpm}{{\rpr}}

\nc{\mmbox}[1]{\mbox{\ #1\ }} \nc{\ann}{\mrm{ann}}
\nc{\Aut}{\mrm{Aut}} \nc{\can}{\mrm{can}}
\nc{\twoalg}{{two-sided algebra}\xspace}
\nc{\colim}{\mrm{colim}}
\nc{\Cont}{\mrm{Cont}} \nc{\rchar}{\mrm{char}}
\nc{\cok}{\mrm{coker}} \nc{\dtf}{{R-{\rm tf}}} \nc{\dtor}{{R-{\rm
tor}}}

\nc{\depth}{{\mrm d}}
\nc{\Div}{{\mrm Div}} \nc{\End}{\mrm{End}} \nc{\Ext}{\mrm{Ext}}
\nc{\Fil}{\mrm{Fil}} \nc{\Frob}{\mrm{Frob}} \nc{\Gal}{\mrm{Gal}}
\nc{\GL}{\mrm{GL}} \nc{\Hom}{\mrm{Hom}} \nc{\hsr}{\mrm{H}}
\nc{\hpol}{\mrm{HP}} \nc{\id}{\mrm{id}} \nc{\im}{\mrm{im}}
\nc{\incl}{\mrm{incl}} \nc{\length}{\mrm{length}}
\nc{\LR}{\mrm{LR}} \nc{\mchar}{\rm char} \nc{\NC}{\mrm{NC}}
\nc{\mpart}{\mrm{part}} \nc{\pl}{\mrm{PL}}
\nc{\ql}{{\QQ_\ell}} \nc{\qp}{{\QQ_p}}
\nc{\rank}{\mrm{rank}} \nc{\rba}{\rm{RBA }} \nc{\rbas}{\rm{RBAs }}
\nc{\rbpl}{\mrm{RBPL}}
\nc{\rbw}{\rm{RBW }} \nc{\rbws}{\rm{RBWs }} \nc{\rcot}{\mrm{cot}}
\nc{\rest}{\rm{controlled}\xspace}
\nc{\rdef}{\mrm{def}} \nc{\rdiv}{{\rm div}} \nc{\rtf}{{\rm tf}}
\nc{\rtor}{{\rm tor}} \nc{\res}{\mrm{res}} \nc{\SL}{\mrm{SL}}
\nc{\Spec}{\mrm{Spec}} \nc{\tor}{\mrm{tor}} \nc{\Tr}{\mrm{Tr}}
\nc{\mtr}{\mrm{sk}}

\nc{\ab}{\mathbf{Ab}} \nc{\Alg}{\mathbf{Alg}}
\nc{\Algo}{\mathbf{Alg}^0} \nc{\Bax}{\mathbf{Bax}}
\nc{\Baxo}{\mathbf{Bax}^0} \nc{\RB}{\mathbf{RB}}
\nc{\RBo}{\mathbf{RB}^0} \nc{\BRB}{\mathbf{RB}}
\nc{\Dend}{\mathbf{DD}} \nc{\bfk}{{\bf k}} \nc{\bfone}{{\bf 1}}
\nc{\base}[1]{{a_{#1}}} \nc{\detail}{\marginpar{\bf More detail}
\noindent{\bf Need more detail!}
\svp}
\nc{\Diff}{\mathbf{Diff}} \nc{\gap}{\marginpar{\bf
Incomplete}\noindent{\bf Incomplete!!}
\svp}
\nc{\FMod}{\mathbf{FMod}} \nc{\mset}{\mathbf{MSet}}
\nc{\rb}{\mathrm{RB}} \nc{\Int}{\mathbf{Int}}
\nc{\Mon}{\mathbf{Mon}}
\nc{\remarks}{\noindent{\bf Remarks: }}
\nc{\OS}{\mathbf{OS}} 
\nc{\Rep}{\mathbf{Rep}}
\nc{\Rings}{\mathbf{Rings}} \nc{\Sets}{\mathbf{Sets}}
\nc{\DT}{\mathbf{DT}}

\nc{\BA}{{\mathbb A}} \nc{\CC}{{\mathbb C}} \nc{\DD}{{\mathbb D}}
\nc{\EE}{{\mathbb E}} \nc{\FF}{{\mathbb F}} \nc{\GG}{{\mathbb G}}
\nc{\HH}{{\mathbb H}} \nc{\LL}{{\mathbb L}} \nc{\NN}{{\mathbb N}}
\nc{\QQ}{{\mathbb Q}} \nc{\RR}{{\mathbb R}} \nc{\BS}{{\mathbb{S}}} \nc{\TT}{{\mathbb T}}
\nc{\VV}{{\mathbb V}} \nc{\ZZ}{{\mathbb Z}}

\nc{\calao}{{\mathcal A}} \nc{\cala}{{\mathcal A}}
\nc{\calc}{{\mathcal C}} \nc{\cald}{{\mathcal D}}
\nc{\cale}{{\mathcal E}} \nc{\calf}{{\mathcal F}}
\nc{\calfr}{{{\mathcal F}^{\,r}}} \nc{\calfo}{{\mathcal F}^0}
\nc{\calfro}{{\mathcal F}^{\,r,0}} \nc{\oF}{\overline{F}}
\nc{\calg}{{\mathcal G}} \nc{\calh}{{\mathcal H}}
\nc{\cali}{{\mathcal I}} \nc{\calj}{{\mathcal J}}
\nc{\call}{{\mathcal L}} \nc{\calm}{{\mathcal M}}
\nc{\caln}{{\mathcal N}} \nc{\calo}{{\mathcal O}}
\nc{\calp}{{\mathcal P}} \nc{\calq}{{\mathcal Q}} \nc{\calr}{{\mathcal R}}
\nc{\calt}{{\mathcal T}} \nc{\caltr}{{\mathcal T}^{\,r}}
\nc{\calu}{{\mathcal U}} \nc{\calv}{{\mathcal V}}
\nc{\calw}{{\mathcal W}} \nc{\calx}{{\mathcal X}}
\nc{\CA}{\mathcal{A}}

\nc{\fraka}{{\mathfrak a}} \nc{\frakB}{{\mathfrak B}}
\nc{\frakb}{{\mathfrak b}} \nc{\frakd}{{\mathfrak d}}
\nc{\oD}{\overline{D}}
\nc{\frakF}{{\mathfrak F}} \nc{\frakg}{{\mathfrak g}}
\nc{\frakm}{{\mathfrak m}} \nc{\frakM}{{\mathfrak M}}
\nc{\frakMo}{{\mathfrak M}^0} \nc{\frakp}{{\mathfrak p}}
\nc{\frakS}{{\mathfrak S}} \nc{\frakSo}{{\mathfrak S}^0}
\nc{\fraks}{{\mathfrak s}} \nc{\os}{\overline{\fraks}}
\nc{\frakT}{{\mathfrak T}}
\nc{\oT}{\overline{T}}
\nc{\frakX}{{\mathfrak X}} \nc{\frakXo}{{\mathfrak X}^0}
\nc{\frakx}{{\mathbf x}}
\nc{\frakTx}{\frakT}      
\nc{\frakTa}{\frakT^a}        
\nc{\frakTxo}{\frakTx^0}   
\nc{\caltao}{\calt^{a,0}}   
\nc{\ox}{\overline{\frakx}} \nc{\fraky}{{\mathfrak y}}
\nc{\frakz}{{\mathfrak z}} \nc{\oX}{\overline{X}}

\font\cyr=wncyr10

\nc{\al}{\alpha}
\nc{\lam}{\lambda}
\nc{\lr}{\longrightarrow}

\nc{\pr}{\prec}
\nc{\su}{\succ}
\nc{\prp}{\prec_P}
\nc{\scp}{\succ_P}
\nc{\pra}{\prec_A}
\nc{\sca}{\succ_A}
\nc{\prb}{\prec_B}
\nc{\scb}{\succ_B}
\nc{\nea}{\nearrow} 
\nc{\se}{\searrow}
\nc{\sw}{\swarrow}
\nc{\nw}{\nwarrow}
\nc{\nep}{\ne^{op}}
\nc{\sep}{\se^{op}}
\nc{\swp}{\sw^{op}}
\nc{\nwp}{\nw^{op}}
\nc{\east}{\succ}
\nc{\west}{\prec}
\nc{\north}{\wedge}
\nc{\south}{\vee}
\nc{\eastt}{\east^t}
\nc{\westt}{\west^t}
\nc{\northt}{\north^t}
\nc{\ob}{{\ \begin{picture}(-1,1)(-1,-3)\circle*{3}\end{picture}\ \,}}

\nc{\dftimes}{\blacksquare}

\nc{\dasib}{differential antisymmetric infinitesimal bialgebra\xspace}
\nc{\dasibs}{differential antisymmetric infinitesimal bialgebra\xspace}
\nc{\Dasib}{Differential antisymmetric infinitesimal bialgebra\xspace}
\nc{\Dasibs}{Differential antisymmetric infinitesimal bialgebras\xspace}
\nc{\dASIb}{differential ASI bialgebra\xspace}
\nc{\DASIb}{Differential ASI bialgebra\xspace}
\nc{\dASIbs}{differential ASI bialgebras\xspace}
\nc{\DASIbs}{Differential ASI bialgebras\xspace}

\nc{\asib}{antisymmetric infinitesimal bialgebra\xspace}
\nc{\ASIb}{ASI bialgebra\xspace}
\nc{\asibs}{antisymmetric infinitesimal bialgebras\xspace}
\nc{\ASIbs}{ASI bialgebras\xspace}

\nc{\da}{differential algebra\xspace}
\nc{\das}{differential algebras\xspace}
\nc{\dca}{differential coalgebra\xspace}
\nc{\dcas}{differential coalgebras\xspace}
\nc{\nva}{noncommutative Novikov algebra\xspace}
\nc{\nvas}{noncommutative Novikov algebras\xspace}
\nc{\nvbia}{noncommutative Novikov bialgebra\xspace}
\nc{\nvbias}{noncommutative Novikov bialgebras\xspace}
\nc{\nvca}{noncocommutative Novikov coalgebra\xspace}
\nc{\nvcas}{noncocommutative Novikov coalgebras\xspace}
\nc{\pnvas}{noncommutative pre-Novikov algebras\xspace}
\nc{\pnva}{noncommutative pre-Novikov algebra\xspace}
\nc{\pnvca}{noncocommutative pre-Novikov coalgebra\xspace}
\nc{\lsa}{\ell_{\prec_A}}
\nc{\rsa}{r_{\prec_A}}
\nc{\lusa}{\ell_{{\succ_A}}}
\nc{\rusa}{r_{{\succ_A}}}
\nc{\lsb}{\ell_{\prec_B}}
\nc{\rsb}{r_{\prec_B}}
\nc{\lusb}{\ell_{{\succ_B}}}
\nc{\rusb}{r_{{\succ_B}}}
\nc{\at}{\cdot_d}

\nc{\wt}{of weight $\lambda$\xspace}
\nc{\pa}{\partial}

\nc{\vsa}{\vspace{-.1cm}}
\nc{\vsb}{\vspace{-.2cm}}
\nc{\vsc}{\vspace{-.3cm}}
\nc{\vsd}{\vspace{-.4cm}}
\nc{\vse}{\vspace{-.5cm}}

\nc{\sh}[1]{\textcolor{purple}{Shanghua:#1}}
\nc{\yz}[1]{\textcolor{green}{Yizhen:#1}}
\nc{\li}[1]{\textcolor{red}{#1}}
\nc{\lir}[1]{\textcolor{blue}{Li:#1}}
\begin{document}
	
\title[A bialgebra theory for weighted differential algebras]{
Noncommutative Novikov bialgebras and differential antisymmetric infinitesimal bialgebras with weight}

\author{Shanghua Zheng}
\address{School of Mathematics and Statistics, Jiangxi Normal University, Nanchang, Jiangxi 330022, China}
\email{zhengsh@jxnu.edu.cn}

\author{Yizhen Li}
\address{School of Mathematics and Statistics, Jiangxi Normal University, Nanchang, Jiangxi 330022, China}
\email{liyz@jxnu.edu.cn}

\author{Liushuting Yang}
\address{School of Mathematics and Statistics, Jiangxi Normal University, Nanchang, Jiangxi 330022, China}
\email{liaostyl@jxnu.edu.cn}

\author{Li Guo}
\address{Department of Mathematics and Computer Science,
	Rutgers University,
	Newark, NJ 07102, USA}
\email{liguo@rutgers.edu}

\begin{abstract}
This paper first develops a bialgebra theory for a noncommutative Novikov algebra, called a noncommutative Novikov bialgebra, which is further characterized by matched pairs and Manin triples of noncommutative Novikov algebras. The  classical Yang-Baxter type equation, $\mathcal{O}$-operators, and noncommutative pre-Novikov algebras are introduced to study noncommutative Novikov bialgebra. As an application,
noncommutative pre-Novikov algebras are obtained from  differential dendriform algebras. Next, to generalize Gelfand's classical construction of a Novikov algebra from a commutative differential algebra to the bialgebra context in the noncommutative case, we establish antisymmetric infinitesimal (ASI) bialgebras for (noncommutative) differential algebras, and obtain the condition under which a differential ASI bialgebra induces a noncommutative Novikov bialgebra.
\end{abstract}

\makeatletter
\@namedef{subjclassname@2020}{\textup{2020} Mathematics Subject Classification}
\makeatother
\subjclass[2020]{
	16T10, 
	12H05, 
	17B38, 
	17D25, 
	16T25,	
	16T15, 
	17A36, 
	17A30. 
}

\keywords{Novikov algebra, differential algebra, antisymmetric infinitesimal bialgebra, Yang-Baxter equation, dendriform algebra}

\maketitle

\vspace{-1cm}

\tableofcontents

\vspace{-1cm}

\allowdisplaybreaks

\section{Introduction}

This paper establishes bialgebra theories for noncommutative Novikov algebras and noncommutative differential algebras, and extends S. Gelfand's classical relation between commutative differential  algebras  and Novikov algebras to the context of noncommutative bialgebras. 

\subsection{Novikov algebras and  differential  algebras}
 A {\bf Novikov algebra} is a vector space $A$ with a binary operation $\circ$ satisfying
\begin{eqnarray}
(x \circ y) \circ z -x \circ (y\circ z) = (y\circ x)\circ z-y\circ (x \circ z),\mlabel{eq:na1}\\
 (x \circ y) \circ z = (x \circ z)\circ y,\quad x,y,z\in A. \mlabel{eq:na2}
\end{eqnarray}
The classical example of a Novikov algebra was given by S. Gelfand~\cite{GD79} from commutative differential  algebras.  
Given a commutative  differential algebra $(A,\cdot, \partial)$ consisting of  a   commutative  algebra $(A,\cdot)$ and a derivation $\pa$ on $A$, 
the binary operation
\begin{equation} \mlabel{eq:xcy}
x\circ y :=x\cdot \partial(y), \quad x,y\in A,
\end{equation}
defines a Novikov algebra $(A,\circ)$. It was shown that each Novikov algebra can be embedded 
into a commutative differential algebra \cite{DL02}. Furthermore,  Eqs.~\meqref{eq:na1}--\meqref{eq:na2} are shown to span all identities that can hold for all commutative  differential algebras via Eq.~\meqref{eq:xcy}~\cite{BCZ17}. 
Novikov algebras were first introduced from
Hamiltonian operators in the formal variational calculus~\cite{BN85} and later appeared in the study of Poisson brackets of hydrodynamic type~\cite{GD79}. There have been numerous studies of Novikov algebras with broad applications to areas such as pre-Lie algebras and Lie conformal algebras~\mcite{BM01,BM02,BCZ17,DL02,K1,KMS24,KN24,KSO19,LH24,LW25,PB11,ZC07,ZGM24}. 
Recently, multi-Novikov algebras arose from regularity structure of stochastic PDEs~\mcite{BD,BHZ}. Novikov type structures on groups and Lie algebras have also been explored~\mcite{GGHZ24}. 

On the other hand, in order to generalize the Gelfand's classical construction of a Novikov algebra from a commutative differential algebra to the noncommutative case, the notion of a \nva was introduced by Sartayev and Kolesnikov in~\cite{SK23}.
\begin{defi}\mlabel{defi:nva}
A {\bf \nva} is a vector space $A$ equipped with binary operations $\succ, \prec$ satisfying the following conditions:
\begin{align}
(x\succ y)\prec z&= x\succ (y\prec z),\mlabel{eq:nv1}\\
(x\prec y)\ob z&= x\ob (y\succ z),\quad x,y,z\in A,\mlabel{eq:nv2}
\end{align}
where $\ob$ is the abbreviation
\begin{equation}
\ob:=\prec+\succ.
\mlabel{eq:ob}
\end{equation}
\end{defi}
In analog to Gelfand's construction, a differential algebra induces a \nva, a property that can be tracked back to the paper of Loday~\cite{Lo1}. 
The latest progress in \nvas  was the construction of monomial basis of free \nvas on a set~\cite{DS25}. 

\subsection{Novikov bialgebras and differential antisymmetric  infinitesimal  bialgebras}
A recent study~\cite{HBG23} introduced the notion of a Novikov bialgebra to construct infinite-dimensional Lie bialgebras via the Novikov bialgebra affinization, lifting the Balinsky-Novikov construction of infinite-dimensional Lie algebras via the Novikov algebra affinization \cite{BN85}.
It was shown that a Novikov bialgebra was simultaneously characterized by a Manin triple of Novikov algebras and a matched pair of Novikov algebras. Furthermore, the concept of a Novikov Yang-Baxter equation was proposed to construct Novikov bialgebras.

On the other hand, in order to  generalize the study of antisymmetric  infinitesimal bialgebras (\asi) to the
 context of differential algebras, the concept of a differential \asi  bialgebra was introduced in \cite{LLB23}.  Moreover, it was proved that Poisson bialgebras can be constructed
 from commutative and cocommutative differential ASI bialgebras with two derivations and two coderivations, 
generalizing the typical construction of Poisson algebras from commutative differential algebras with two  derivations.

Most recently, the relations between   differential \asi bialgebras and Novikov bialgebras were established in \cite{HBG24}.  
It was shown that under some additional conditions,  every  commutative and cocommutative differential ASI bialgebra derived a Novikov bialgebra.  
Inspired by this derivation, this paper also takes into account the derived structures of differential \asi bialgebras. 

\vsd

\subsection{Bialgebras for noncommutative Novikov algebras and differential algebras}
The purpose of this paper is to develop a bialgebra theory for both a noncommutative Novikov algebra and a differential algebra by the Manin triple approach, and lift the above-mentioned connection between differential algebras and noncommutative Novikov algbras to the level of bialgebras. 

Similar to Novikov bialgebras~\mcite{HBG24}, we treat the noncommutative case by first characterizing it by Manin triples and matched pairs of noncommutative Novikov algebras. Antisymmetric solutions of Novikov
Yang-Baxter equations (NYBEs) in noncommutative Novikov algebras naturally leads to noncommutative Novikov bialgebras. In turn, such solutions
are provided by $\mathcal{O}$-operators of noncommutative Novikov algebras, which can further be obtained from noncommutative pre-Novikov algebras. In addition,  quasi-Frobenius noncommutative Novikov algebras
provide nondegenerate solutions of the NYBE. Furthermore, a bialgebra theory of differential algebras of weight $\lambda$,  including their characterization by double constructions of differential Frobenius algebras of weight $\lambda$ and matched pairs of differential algebras of weight $\lambda$, are established. We show that under suitable conditions, every differential \asi bialgebra of weight $0$ induces a  noncommutative Novikov bialgebra.
In summary, these connections are depicted in the following diagram.
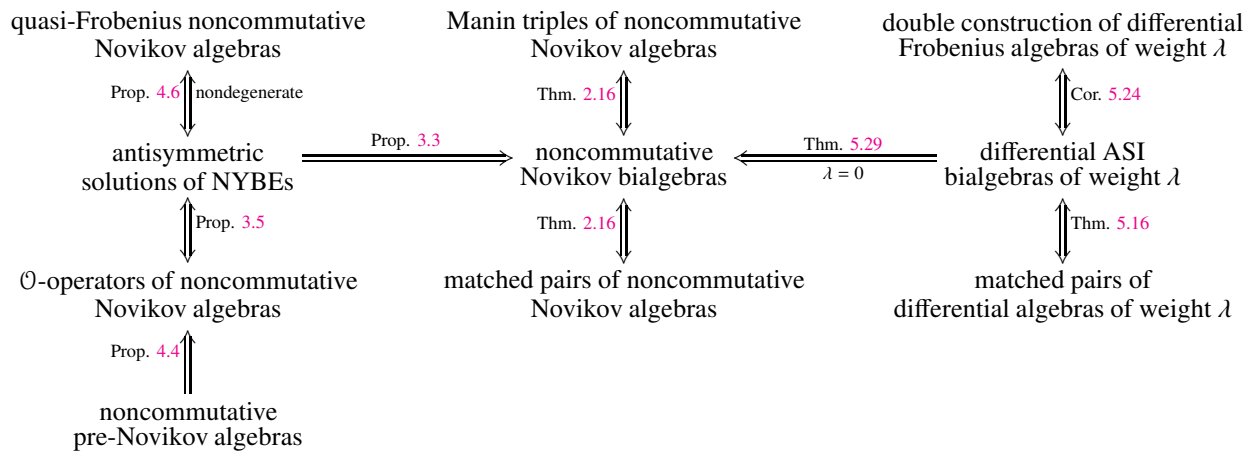
\vsa
$$\footnotesize{
\xymatrix{
\txt{quasi-Frobenius noncommutative\\ Novikov algebras} &\txt{Manin triples of noncommutative\\ Novikov algebras}&\txt{double construction of differential\\ Frobenius algebras of weight $\lambda$}\\
 \ar@2{<->}^{\text{Prop.}~\ref{pro:4.6}}[u]_{\text{nondegenerate}}
\txt{antisymmetric \\solutions of NYBEs}
\ar@2{->}^{\text{Prop.}~\ref{pro:antieybe}}[r]\ar@2{<->}^{\text{Prop.}~\ref{pro:oelda}}[d]&\txt{noncommutative \\ Novikov bialgebras}    \ar@2{<->}_{\text{Thm.}~\ref{thm:mpmt}}[d] \ar@2{<->}^{\text{Thm.}~\ref{thm:mpmt}}[u]&\txt{differential \asi \\ bialgebras of weight $\lambda$} \ar@2{->}_{\text{Thm.}~\ref{thm:main}}[l]^{\text{$\lambda=0$}} \ar@2{<->}_{\text{Cor.}~\ref{cor:3.13}}[u]
\ar@2{<->}^{\text{Thm.}~\ref{thm:3.5}}[d]
\\
\txt{$\mathcal{O}$-operators of noncommutative\\ Novikov algebras} &\txt{matched pairs of noncommutative\\ Novikov algebras}&\txt{matched pairs of\\ \das of weight $\lambda$}
\\
\txt{noncommutative \\ pre-Novikov algebras}
\ar@2{->}^{\text{Prop.}~\ref{pro:id}}[u]
&&}}
$$

Here is the layout of the paper. 
Section~\mref{sec:nbialg} gives the notion of a noncommutative Novikov bialgebra, and its characterizations  by a matched pair and by a Manin triple of noncommutative Novikov algebras.

To construct noncommutative Novikov bialgebras, Section~\mref{sec:NYBE} introduces the notion of the NYBE in noncommutative Novikov algebras. $\mathcal{O}$-operators of noncommutative Novikov algebras are proposed to give antisymmetric solutions of the NYBE in the semi-product noncommutative Novikov algebras.

Section~\mref{sec:quasiFor} gives the notion of a \pnva.  We establish relations between  $\mathcal{O}$-operators of \nvas   and  \pnvas.  We then prove that quasi-Frobenius noncommutative Novikov algebras are equivalent to  antisymmetric nondegenerate solutions of the NYBE.

To generalize differential \asi bialgebras given by \cite{LLB23} to the arbitrary weight case, Section~\mref{sec:DASI}  presents  the notion of a differential \asi bialgebras of weight $\lambda$, as well as examples of differential \asi bialgebras of weight $\lambda$. 
Furthermore, differential \asi bialgebras \wt are  equivalent to a  matched pair of differential algebras \wt and a double construction of  differential Frobenius algebra \wt. In the special case of $\lambda=0$,  every differential \asi bialgebra derives a \nvbia under compatibility conditions.

\noindent
{\bf Notations: }
\begin{enumerate}
\item Throughout this paper, we fix a field $\bfk$ of characteristic $0$. All vector spaces, tensor products, and
linear homomorphisms are over $\bfk$. All the vector spaces and algebras are finite dimensional,  although some notions and statements still hold in the infinite-dimensional case.
By an algebra, we mean an associative algebra that is not necessarily commutative or unitary.
\item We denote by $V^*$  the dual space of a vector space $V$. The notation $\langle\,,\, \rangle : V^*\times V\to  \bfk$ given by
	$\,\langle u^*, v \rangle :=u^*(v), \,u^*\in V^*, v\in V, $
	denotes the  natural pair between $V^*$ and $V$. For a linear map $\psi :  V\to  W$, the linear dual of $\psi$ is defined by
	\begin{equation}
		\psi^*: W^*\to  V^*,\quad\langle \psi^*(w^*), v \rangle: =\langle w^*, \psi(v) \rangle , \quad v\in V, w^*\in W^*.
		\mlabel{eq:phis}
	\end{equation}
\item Let $A$ and $V$ be vector spaces and $\rho:A\rightarrow\mathrm{End}(V)$ be a linear map. We define the linear map $\rho^{*}:A\rightarrow\mathrm{End}(V^{*})$ by
\begin{equation*}
	\langle \rho^{*}(x)u^{*},v\rangle=\langle \big( \rho(x) \big)^{*}u^{*},v\rangle=\langle u^{*},\rho(x)v\rangle,\quad x\in A, u^{*}\in V^{*},v\in V.
\end{equation*}
\item For a vector space $A$ with a binary operation $\circ:A\otimes A\rightarrow A$, we denote $L_{\circ},R_{\circ}:A\rightarrow A$ by
\begin{equation*}
L_{\circ}(x)y=x\circ y=R_{\circ}(y)x,\quad x,y\in A.
\end{equation*}
Furthermore,  denote by $\sigma$ the flip linear map on $A\ot A$, that is, $\sigma (x\ot y)=y\ot x$.
\end{enumerate}

\section{Noncommutative Novikov bialgebras}
\mlabel{sec:nbialg}
This section first defines the representation of a \nva, which can be equivalently characterized by the semi-product noncommutative Novikov algebra.  It then introduces the notions of a \nvbia, a matched pair of noncommutative Novikov algebras and a Manin triple of noncommutative Novikov algebras. The latter two are used to characterize  \nvbias.

\subsection{Representations and matched pairs of noncommutative Novikov algebras}
We give the notion of a representation of noncommutative Novikov algebras.
\begin{defi}\mlabel{defi:5.3}
Let $A:=(A,\prec,\succ)$ be a \nva. A {\bf representation} of $A$ is a quintuple $(V, \ell_\prec, r_\prec, \ell_\succ, r_\succ)$ where $V$ is a vector space and $\ell_\prec, r_\prec, \ell_\succ, r_\succ: A\to \End(V)$ are linear maps satisfying the following identities. 
\begin{align}
\ell_\prec(x\succ y)&=\ell_\succ(x)\ell_\prec(y),\mlabel{eq:r1}\\
r_\prec(x)\ell_\succ(y)&=\ell_\succ(y)r_\prec(x),\mlabel{eq:r2}\\
r_\prec(x)r_\succ(y)&=r_\succ (y\prec x),\mlabel{eq:r3}\\
\ell_\ob(x\prec y)&=\ell_\ob(x)\ell_\succ(y),\mlabel{eq:r4}\\
r_\ob(x)\ell_\prec(y)&=\ell_\ob(y)r_\succ(x),\mlabel{eq:r5}\\
r_\ob(x)r_\prec(y)&=r_\ob(y\succ x),\quad x,y\in A.\mlabel{eq:r6}
\end{align}
\end{defi}

\begin{pro}\mlabel{pro:dg}Let  $A:=(A,\prec,\succ)$ be a \nva. If $(V,\ell_\prec, r_\prec,\ell_\succ, r_\succ)$ is a representation of $A$, then $(V^*,-r^*_\ob, \ell^*_\succ, r^*_\prec, -\ell^*_\ob)$ is also a representation of $A$.
\end{pro}
\begin{proof}For all $x,y,z\in A$ and $a^*\in A^*$, we have
\begin{eqnarray*}
\langle-r^*_\ob(x\succ y)a^*,z\rangle=\langle a^*, -r_\ob (x\succ y)z\rangle,\quad\langle r^*_\prec (x)(-r^*_\ob (y))a^*,z \rangle=\langle r^*_\prec (x)(-r^*_\ob (y))a^*,(-r_\ob(y)r_\prec(x))z \rangle.
\end{eqnarray*}
Then by Eq.~\meqref{eq:r6} for the representation $V$, we get $-r^*_\ob(x\succ y)=r^*_\prec (x)(-r^*_\ob (y))$. This verifies the identity \meqref{eq:r1} for the representation $V^*$. The other identities can be verified in the similar way.
\end{proof}

\begin{ex} \mlabel{ex:ap}
Let $A:=(A,\prec,\succ)$ be a \nva.  Then $(A, L_\prec, R_\prec, L_\succ, R_\succ)$ is a representation of $A$, called the {\bf adjoint representation} of $A$.
Moreover, $(A^*, -R^*_\ob, L^*_\succ, R^*_\prec, -L^*_\ob)$ is also a representation of $A$, called the {\bf coadjoint representation} of $A$.
\end{ex}

\begin{defi}	\mlabel{defi:mp}
Let $A:=(A,\pra,\sca)$ and $B:=(B,\prb,\scb)$ be \nvas.
Let $\lsa$,$\rsa$,$\lusa$,$\rusa:A\to \End(B)$ and $\lsb,\rsb,\lusb,\rusb: B\to \End(A)$ be linear maps. Define  binary operations $\prec,\succ$ on the direct sum $A\oplus B$ by
\begin{eqnarray}
(x+a)\prec_{A\bowtie B} (y+b):=(x\prec_A y+\rsb(b)x+\lsb(a)y)+(a\prec_B b+\lsa(x)b+\rsa(y)a),\mlabel{eq:do1}\\
(x+a)\succ_{A\bowtie B} (y+b):=(x\sca y+\rusb(b)x+\lusb(a)y)+(a\scb b+\lusa(x)b+\rusa(y)a),  \mlabel{eq:do2}
\end{eqnarray}
for all $x,y\in A$ and $ a,b\in B$. If $(A\oplus B, \prec_{A\bowtie B},\succ_{A\bowtie B})$ is  a \nva, denoted by $A\bowtie_{\lsb,\lsa,\lusb,\lusa}^{\rsb,\rsa,\rusb,\rusa} B$ or simply $A\bowtie B$,
		then the decuple
		$$(A,B,\lsa,\rsa,\lusa,\rusa,\lsb,\rsb,\lusb,\rusb)$$
		is called a {\bf matched pair of \nvas}.
\end{defi}
\begin{pro}\mlabel{pro:ld-rep}
Let $A:=(A,\prec,\su)$ be a \nva and let $V$ be a vector space. For linear maps $\ell_{\prec},r_{\prec},\ell_{\su}, r_{\su}: A\to \End(V)$, define  binary operations $\prec',\su'$ on $A\oplus V$ by
	\begin{align*}
		&(x+u)\prec'(y+v):=x\prec y+\ell_{\prec}(x)v+r_{\prec}(y)u,\\
	&(x+u)\su'(y+v):=x\su y+\ell_{\su}(x)v+r_{\su}(y)u,\quad x,y\in A, u,v\in V.
	\end{align*}
Then $(A\oplus V,\prec',\su')$ is  a \nva if and only if
$(V,\ell_{\prec},r_{\prec},\ell_{\su}, r_{\su})$ is a  representation of $(A,\prec,\su)$.
In this case,  $(A\oplus V, \prec',\su')$ is called the {\bf semi-direct product  \nva} of $(A,\prec,\su)$,  denoted  by $(A\ltimes_{r_{\prec},r_{\su}}^{\ell_{\prec},\ell_{\su}}V,\prec',\su')$.
\end{pro}

\begin{lem}\mlabel{lem:5.8}
Let $A:=(A,\prec_A,\sca)$ and $B:=(B,\prec_B,\scb)$ be \nvas. Let $\lsa$,$\rsa$,$\lusa$,$\rusa :A\to \End(B)$ and $\lsb,\rsb,\lusb,\rusb: B\to \End(A)$ be linear maps.  Then the decuple $$(A,B,\lsa,\rsa,\lusa,\rusa,\lsb,\rsb,\lusb,\rusb)$$
is a matched pair of \nvas if and only if
\begin{enumerate}
\item $(B,\lsa,\rsa,\lusa,\rusa)$ is a representation of $A$,
\item$(A,\lsb,\rsb,\lusb,\rusb)$ is a representation of $B$, and
\item the following compatibility conditions hold: for all $x,y\in A$ and $a,b\in B$,
\begin{align}
\lusb(a)(x\prec_Ay)&=(\lusb(a)x)\prec_Ay+\lsb(\rusa(x)a)y,\mlabel{eq:5.81}\\
\lusa(x)(a\prec_Bb)&=(\lusa(x)a)\prec_Bb+\lsa(\rusb(a)x)b,\mlabel{eq:5.82}\\
\rsb(a)(x\succ_Ay)&=x\succ_A(\rsb(a)y)+\rusb(\lsa(y)a)x,\mlabel{eq:5.83}\\
\rsa(x)(a\succ_Bb)&=a\succ_B(\rsa(x)b)+\rusa(\lsb(b)x)a,\mlabel{eq:5.84}\\
(\rusb(a)x)\prec_Ay+\lsb(\lusa(x)a)y&=x\succ_A(\lsb(a)y)+\rusb(\rsa(y)a)x,\mlabel{eq:5.85}\\
(\rusa(x)a)\prec_Bb+\lsa(\lusb(a)x)b&=a\succ_B(\lsa(x)b)+\rusa(\rsb(b)x)a,\mlabel{eq:5.86}\\
\ell_{\ob_B}(a)(x\succ_Ay)&=(\lsb(a)x)\ob_Ay+\ell_{\ob_B}(\rsa(x)a)y,\mlabel{eq:5.87}\\
\ell_{\ob_A}(x)(a\succ_Bb)&=(\lsa(x)a)\ob_Bb+\ell_{\ob_A}(\rsb(a)x)b,\mlabel{eq:5.88}\\
r_{\ob_B}(a)(x\prec_Ay)&=x\ob_A(\rusb(a)y)+r_{\ob_B}(\lusa(y)a)x,\mlabel{eq:5.89}\\
r_{\ob_A}(x)(a\prec_Bb)&=a\ob_B(\rusa(x)b)+r_{\ob_A}(\lusb(b)x)a,\mlabel{eq:5.810}\\
(\rsb(a)x)\ob_Ay+\ell_{\ob_B}(\lsa(x)a)y&=x\ob_A(\lusb(a)y)+r_{\ob_B}(\rusa(y)a)x,\mlabel{eq:5.811}\\
(\rsa(x)a)\ob_Bb+\ell_{\ob_A}(\lsb(a)x)b&=a\ob_B(\lusa(x)b)+r_{\ob_A}(\rusb(b)x)a.\mlabel{eq:5.812}
\end{align}
 \end{enumerate}
\end{lem}
\begin{proof}
The proof follows from a direct calculation.
\end{proof}

\subsection{Noncocommutative Novikov bialgebras}

As the dual notion, we define a  \nvca.
\begin{defi}\mlabel{defi:nvca}
A \textbf{\nvca} is a quadruple $(A,\Delta_{\prec},\Delta_{\su})$ consisting of a vector space $A$ and  linear maps $\Delta_{\prec},\Delta_{\su}:A\to A\ot A$ such that,
\begin{align}
(\Delta_{\succ}\otimes \id)\Delta_{\prec}(x)&=(\id\otimes \Delta_{\prec})\Delta_{\succ}(x),\mlabel{eq:exca1}\\
(\Delta_{\prec}\otimes \id)\Delta_{\ob}(x)&=(\id\otimes \Delta_{\succ})\Delta_{\ob}(x), \quad x\in A.\mlabel{eq:exca2}
\end{align}
Here $\Delta_\ob:=\Delta_\prec+\Delta_\succ$.
\end{defi}
\begin{defi}~\cite{KUZ22,HBG23}
A {\bf Novikov coalgebra} is a vector space $A$ equipped with a linear map $\Delta: A\to A\ot A$
such that for all $x\in A$,
\begin{align}(\Delta\ot\id)\Delta(x)&=(\sigma\ot\id)(\id\ot\Delta)\sigma\Delta(x)
,\mlabel{eq:co1}\\
(\id\ot \Delta)\Delta(x)-(\sigma\ot \id)(\id\ot \Delta)\Delta(x)&=(\Delta\ot\id)\Delta(x)-(\sigma\ot\id)(\Delta\ot \id)\Delta(x).\mlabel{eq:co2}
\end{align}
\end{defi}
For a given \nvca $A:=(A,\Delta_\prec,\Delta_\succ)$, if $\Delta$ is cocommutative in the sense of $\Delta_\succ =\sigma\Delta_\prec$, then $A$ is exactly a  Novikov coalgebra.

\begin{lem}\mlabel{lem:nva-equi}
	Let $A$ be a vector space, and let $\Delta_{\prec_{A}},\Delta_{\su_{A}}: A\to A\ot A$ be linear maps. Let $\prec_{A^{*}},\su_{A^{*}}:A^{*}\otimes A^{*}\rightarrow A^{*}$ be the linear duals of $\Delta_{\prec_{A}}$ and $\Delta_{\su_{A}}$, respectively.
Then $(A,\Delta_{\prec_{A}},\Delta_{\su_{A}})$ is  a \nvca if and only if $(A^*,\prec_{A^*},\su_{A^*})$ is a \nva.
\end{lem}

\begin{defi}\mlabel{def:lpdbia}
A \textbf{\nvbia} is a quintuple $(A,\prec,\su,\Delta_{\prec},\Delta_{\su})$ consisting of a vector space $A$ and linear maps
$$\prec,\su : A\ot A \to A , \Delta_{\prec},\Delta_{\su}: A \to  A\ot A $$
such that
\begin{enumerate}
\item $(A,\prec,\su)$ is a \nva,
\item $(A,\Delta_{\prec},\Delta_{\su})$ is a \nvca, and
\item the following compatibility conditions hold: for $x,y \in A$,
\begin{align}
\Delta_{\prec}(x\prec y)&=(R_{\prec}(y)\ot \id)\Delta_{\prec}(x)+(\id \ot L_{\ob}(x))\Delta_{\ob}(y),\mlabel{eq:nvb1}\\
\Delta_{\su}(x\su y)&=(R_{\ob}(y)\ot \id)\Delta_{\ob}(x)+(\id \ot L_{\su}(x))\Delta_{\su}(y),\mlabel{eq:nvb2}\\
(\id \ot R_{\prec}(y))\Delta_{\ob}(x)+\sigma((\id\ot R_{\prec}(x))\Delta_{\ob}(y))&=(L_{\su}(y)\ot \id)\Delta_{\ob}(x)+\sigma((L_{\su}(x)\ot \id)\Delta_{\ob}(y)),\mlabel{eq:nvb3}\\
(L_{\ob}(y)\ot \id)\Delta_{\prec}(x)+\sigma((L_{\ob}(x)\ot \id)\Delta_{\prec}(y))&=(\id \ot R_{\ob}(y))\Delta_{\su}(x)+\sigma((\id\ot R_{\ob}(x))\Delta_{\su}(y)).\mlabel{eq:nvb4}
\end{align}
\end{enumerate}
\end{defi}
We next recall  the notion of  a  Novikov bialgebra.
\begin{defi}~\cite[Definition~2.5]{HBG23}
A {\bf Novikov bialgebra} is a tripe $(A,\cdot,\Delta)$, where $(A,\cdot)$ is a  Novikov algebra and $(A,\Delta)$ is a  Novikov coalgebra such that, for all $x,y\in A$, the following conditions are satisfied.
\begin{eqnarray}
\Delta(x\cdot y)=(R(y)\ot \id)\Delta(x)+(\id\ot L(x))(\Delta(y)+\sigma\Delta(y)).\mlabel{eq:cc1}\\
(L_\ob(x)\ot\id)\Delta(y)-(\id\ot L_\ob(x))\sigma\Delta(y)=(L_\ob(y)\ot\id)\Delta(x)-(\id\ot L_\ob(y))\sigma\Delta(x).\mlabel{eq:cc2}\\
(\id\ot R(x)-R(x)\ot\id)(\Delta(y)+\sigma\Delta(y))=(\id\ot R(y)-R(y)\ot\id)(\Delta(x)\ot \sigma \Delta(x)).\mlabel{eq:cc3}
\end{eqnarray}
\end{defi}
\begin{rmk}
If the \nvbia $(A,\prec, \succ,\Delta_\prec,\Delta_\succ)$ satisfies the commutative law (i.e. $x\prec y=y\succ x$), and the cocommutative law (i.e. $\Delta_\succ=\sigma\Delta_\prec$), then Eq.~\meqref{eq:nvb1} is equivalent to Eq.~\meqref{eq:nvb2}, which is just Eq.~\meqref{eq:cc1}, and further Eqs.~\meqref{eq:nvb3} and \meqref{eq:nvb4} are exactly Eqs.~\meqref{eq:cc2} and \meqref{eq:cc3}.
\end{rmk}

\begin{ex}
Let $(A=\bfk e_1\oplus \bfk e_2,\prec,\su)$ be a two-dimensional vector space with binary operations $\prec,\su$ given by the following Cayley tables.

\begin{center}
\begin{tabular}{c | c c c }
$\prec$	& $e_1$ & $e_2$   \\
\hline 	
	$e_1$ & $-e_1$ & $0$  \\
	$e_2$ & $0$ & $-e_2$  \\
\end{tabular}\quad\quad\quad
\begin{tabular}{c | c c c }
$\su$	& $e_1$ & $e_2$   \\
\hline 	
	$e_1$ & $e_1$ & $0$  \\
	$e_2$ & $0$ & $e_2$  \\
\end{tabular}
\end{center}
Define $\Delta_\prec,\Delta_{\su}:A\ot A \rightarrow A$ by
\begin{eqnarray*}
\Delta_\prec(e_1) =-e_1\ot e_1, \quad \Delta_\su(e_1) =e_1\ot e_1,\\
\Delta_\prec(e_2) =-e_2\ot e_2, \quad \Delta_\su(e_2) =e_2\ot e_2.
\end{eqnarray*}
Then it can be checked that $(A,\prec,\su,\Delta_{\prec},\Delta_{\su})$ is a \nvbia.
\end{ex}

\subsection{Characterizations of \nvbias}
We next introduce the notion of a quadratic \nva as follows.
\begin{defi}
A {\bf quadratic \nva} $(A,\prec,\succ,\mathfrak{B})$ is a \nva $(A,\prec,\succ)$ together with a nondegenerate symmetric bilinear form $$\mathfrak{B}: A\otimes A\to \bfk,$$ 
which is {\bf invariant} on $(A,\prec,\succ)$ in the sense that $\frakB$ satisfies
\begin{equation}
\frakB(x\prec y,z)=-\frakB(y,z\ob x)=\frakB(x,y\succ z),\quad x,y,z \in A.\mlabel{eq:inv}
\end{equation}
\end{defi}

\begin{defi}
	Let $A:=(A,\prec_{A},\su_{A})$ be a \nva.   Suppose that  $A^*:=(A^*,\prec_{A^{*}},\su_{A^{*}})$ is also a \nva. If there is a \nva structure
	$(A\oplus A^{*},\prec_{A\bowtie A^*},\su_{A\bowtie A^*})$
	 on $A\oplus A^*$, which contains $(A,\prec_{A},\su_{A})$ and $(A^*,\prec_{A^{*}},\su_{A^{*}})$ as subalgebras, and such that the natural nondegenerate symmetric bilinear form $\frakB_d$ on $A\oplus A^*$ given by
	\begin{equation}
		\frakB_d(x+a^*,y+b^*):=\langle a^*,y\rangle+\langle x,b^*\rangle,\mlabel{eq:sn}
	\end{equation}
	is invariant, then  $\big( (A\oplus A^{*},\prec_{A\bowtie A^*},\su_{A\bowtie A^*},\frak B_{d}), (A,\prec_{A},\su_{A}), (A^*,\prec_{A^{*}},\su_{A^{*}})   \big)$ or simply  $\big( (A\bowtie A^{*},\prec_{A\bowtie A^*},\su_{A\bowtie A^*},\frak B_{d}),  A, A^* \big)$ is called a {\bf Manin triple of \nvas}.
\end{defi}

\begin{thm}\mlabel{thm:mpmt}
Let $(A,\prec_{A},\su_{A})$ and $ (A^*,\prec_{A^{*}},\su_{A^{*}})$ be \nvas. Let $\Delta_{\prec_{A}},\Delta_{\su_{A}}$ be the linear dual of $\prec_{A^{*}},\su_{A^{*}}$, respectively. 
Then the following statements are equivalent:
\begin{enumerate}
\item \mlabel{it:m1}
There is a Manin triple of \nvas $\big( (A\oplus A^{*},\prec_{A\bowtie A^*},\su_{A\bowtie A^*},\frak B_{d})$, $(A,\prec_{A},\su_{A}), (A^*,\prec_{A^{*}},\su_{A^{*}})\big)$.
\item \mlabel{it:m2}
The decuple $(A,A^*,-R^*_{\ob_A}, L^*_{\succ_A}, R^*_{\prec_A}, -L^*_{\ob_A},
			-R^*_{\ob_{A^*}}, L^*_{\succ_{A^*}}, R^*_{\prec_{A^*}}, -L^*_{\ob_{A^*}})$ is a matched pair of \nvas.
\item\mlabel{it:m3}
The quintuple $(A,\prec_{A},\su_{A},\Delta_{\prec_{A}},\Delta_{\su_{A}})$ is a \nvbia.
\end{enumerate}
\end{thm}
\begin{proof}
((a)$\Rightarrow$(b))
	Suppose that  $\big( (A\oplus A^{*},\prec_{A\bowtie A^*},\su_{A\bowtie A^*},\frak B_{d})$, $(A,\prec_{A},\su_{A}), (A^*,\prec_{A^{*}},\su_{A^{*}})\big)$ is a Manin triple of \nvas. For all $x,y,z\in A,  a^*,b^*,c^*\in A^*$, we have 
	\begin{eqnarray*}
	\langle x \su_{A\bowtie A^*} b^*,z\rangle+ \langle    x \su_{A\bowtie A^*} b^*, c^*\rangle &=& \frakB_d(x \su_{A\bowtie A^*} b^*,z+c^*) \\&=& \frakB_d(x \su_{A\bowtie A^*} b^*,z)+\frakB_d(x \su_{A\bowtie A^*} b^*,c^*) \\&=& \frakB_d(z \prec_A x,b^*)-\frakB_d(x ,b^*\ob_{A^*}c^*)\\&=&\langle z,R_{\prec_A}^*(x)b^* \rangle- \langle    c^*, L_{\ob_{A^*}}^*(b^*)x\rangle.
	\end{eqnarray*}
Thus, $x \su_{A\bowtie A^*} b^*=R_{\prec_A}^*(x)b^*-L_{\ob_{A^*}}^*(b^*)x.$
Similarly, we have
\begin{eqnarray*}
		\langle a^* \su_{A\bowtie A^*} y,z\rangle+ \langle    a^* \su_{A\bowtie A^*} y, c^*\rangle &=& \frakB_d(a^* \su_{A\bowtie A^*} y,z+c^*) \\&=& \frakB_d(a^* \su_{A\bowtie A^*} y,z)+\frakB_d(a^* \su_{A\bowtie A^*} y,c^*) \\&=& -\frakB_d(a^*,y\ob_{A}z)+\frakB_d(c^*\prec_{A^*}a^* ,y)\\&=&-\langle z,L_{\ob_{A}}^*(y)a^* \rangle+ \langle    c^*, R_{\prec_{A^*}}^*(a^*)y\rangle.
	\end{eqnarray*}
Hence, $a^* \su_{A\bowtie A^*} y=-L_{\ob_{A}}^*(y)a^*+R_{\prec_{A^*}}^*(a^*)y.$
So we get
\begin{equation}\mlabel{eq:xa1}
(x+a^*) \su_{A\bowtie A^*} (y+b^*)=(x\su y-L_{\ob_{A^*}}^*(b^*)x+R_{\prec_{A^*}}^*(a^*)y)+(a^*\su_{A^*}b^*+R_{\prec_A}^*(x)b^*-L_{\ob_{A}}^*(y)a^*).
\end{equation}
Analogously, we obtain
\begin{equation}\mlabel{eq:xa2}
(x+a^*) \prec_{A\bowtie A^*} (y+b^*)=(x\prec y+L_{\su_{A^*}}^*(b^*)x-R_{\ob_{A^*}}^*(a^*)y)+(a^*\prec_{A^*}b^*-R_{\ob_A}^*(x)b^*+L_{\su_{A}}^*(y)a^*).
\end{equation}
By Definition \mref{defi:mp} and Eqs.~\meqref{eq:xa1}--\meqref{eq:xa2},
$(A,A^*,-R^*_{\ob_A}, L^*_{\succ_A}, R^*_{\prec_A}, -L^*_{\ob_A},
	-R^*_{\ob_{A^*}}, L^*_{\succ_{A^*}}, R^*_{\prec_{A^*}}, -L^*_{\ob_{A^*}})$ is a matched pair of \nvas.
\smallskip

\noindent
((b)$\Rightarrow$(a))
	If $(A,A^*,-R^*_{\ob_A}, L^*_{\succ_A}, R^*_{\prec_A}, -L^*_{\ob_A},
	-R^*_{\ob_{A^*}}, L^*_{\succ_{A^*}}, R^*_{\prec_{A^*}}, -L^*_{\ob_{A^*}})$ is a matched pair of \nvas, then by Definition \mref{defi:mp}, $(A\oplus A^{*},\prec_{A\bowtie A^*},\su_{A\bowtie A^*})$ is a noncommutative Novikov algebra. For all $x,y,z\in A, a^*,b^*,c^*\in A^*$, we have
	\begin{eqnarray*}
		&& \frakB_d((x+a^*) \prec_{A\bowtie A^*} (y+b^*),z+c^*)\\&=&\frakB_d((x\prec_{A} y+L_{\su_{A^*}}^*(b^*)x-R_{\ob_{A^*}}^*(a^*)y)+(a^*\prec_{A^*}b^*-R_{\ob_A}^*(x)b^*+L_{\su_{A}}^*(y)a^*),z+c^*)\\&=&\langle x\prec_{A} y+L_{\su_{A^*}}^*(b^*)x-R_{\ob_{A^*}}^*(a^*)y ,c^*\rangle+ \langle a^*\prec_{A^*}b^*-R_{\ob_A}^*(x)b^*+L_{\su_{A}}^*(y)a^*,z\rangle\\&=&\langle x\prec_{A} y, c^*\rangle+\langle x,b^*\su_{A^*}c^* \rangle-\langle y,c^*\ob_{A^*} a^*\rangle+\langle a^*\prec_{A^*} b^*,z\rangle-\langle b^*,z\ob_{A} x \rangle+\langle a^*,y\su_{A} z\rangle.
		\end{eqnarray*}
			\begin{eqnarray*}
			&& -\frakB_d(y+b^*,(z+c^*) \ob_{A\bowtie A^*} (x+a^*))\\&=&-\frakB_d(y+b^*,z\ob_{A} x+(L_{\su_{A^*}}^*-L_{\ob_{A^*}}^*)(a^*)z+(R_{\prec_{A^*}}^*-R_{\ob_{A^*}}^*)(c^*)x+c^*\ob_{A^*}a^*+(R_{\prec_A}^*-R_{\ob_A}^*)(z)a^*\\&&+(L_{\su_{A}}^*-L_{\ob_{A}}^*)(x)c^*)\\&=&-\frakB_d(y+b^*,z\ob_{A} x-L_{\prec_{A^*}}^*(a^*)z-R_{\su_{A^*}}^*(c^*)x+c^*\ob_{A^*}a^*-R_{\su_A}^*(z)a^*-L_{\prec_{A}}^*(x)c^*)\\&=&-\langle b^*,z\ob_{A} x-L_{\prec_{A^*}}^*(a^*)z-R_{\su_{A^*}}^*(c^*)x\rangle -\langle y,c^*\ob_{A^*}a^*-R_{\su_A}^*(z)a^*-L_{\prec_{A}}^*(x)c^* \rangle \\&=&-\langle b^*,z\ob_{A} x \rangle+\langle a^*\prec_{A^*} b^*,z\rangle+\langle b^*\su_{A^*}c^* ,x\rangle-\langle y,c^*\ob_{A^*} a^*\rangle+\langle a^*,y\su_{A} z\rangle+\langle x\prec_{A} y, c^*\rangle.
		\end{eqnarray*}
		\begin{eqnarray*}
			&& \frakB_d((y+b^*) \su_{A\bowtie A^*} (z+c^*),x+a^*)\\&=&\frakB_d((y\su_{A} z-L_{\ob_{A^*}}^*(c^*)y+R_{\prec_{A^*}}^*(b^*)z)+(b^*\su_{A^*}c^*+R_{\prec_A}^*(y)c^*-L_{\ob_{A}}^*(z)b^*),x+a^*)\\&=&\langle y\su_{A} z-L_{\ob_{A^*}}^*(c^*)y+R_{\prec_{A^*}}^*(b^*)z ,a^*\rangle+ \langle b^*\su_{A^*}c^*+R_{\prec_A}^*(y)c^*-L_{\ob_{A}}^*(z)b^*,x\rangle\\&=&\langle y\su_{A} z,a^* \rangle-\langle y,c^*\ob_{A^*} a^*\rangle + \langle z,a^*\prec_{A^*} b^*\rangle +\langle b^*\su_{A^*}c^* ,x \rangle+ \langle c^*,x\prec_{A} y \rangle-\langle b^*,z\ob_{A} x \rangle.
		\end{eqnarray*}
		Thus, $$\frakB_d((x+a^*) \prec_{A\bowtie A^*} (y+b^*),z+c^*)=-\frakB_d(y+b^*,(z+c^*) \ob_{A\bowtie A^*} (x+a^*))=\frakB_d((y+b^*) \su_{A\bowtie A^*} (z+c^*),x+a^*).$$
 So $\frakB_d$ is invariant, and thus $\big( (A\bowtie A^{*},\prec_{A\bowtie A^*},\su_{A\bowtie A^*},\frak B_{d})$, $(A,\prec_{A},\su_{A}), (A^*,\prec_{A^{*}},\su_{A^{*}})\big)$ is a Manin triple of \nvas.
\smallskip

\noindent
((b)$\Rightarrow$(c))
If	$(A,A^*,-R^*_{\ob_A}, L^*_{\succ_A}, R^*_{\prec_A}, -L^*_{\ob_A},
		-R^*_{\ob_{A^*}}, L^*_{\succ_{A^*}}, R^*_{\prec_{A^*}}, -L^*_{\ob_{A^*}})$ is a matched pair of \nvas, then by a direct computation, for any $x,y \in A, a^*,b^* \in A^*$,
we obtain the following equivalences.
\begin{enumerate}
\item
Eq. \meqref{eq:5.81}$\Leftrightarrow $Eq. \meqref{eq:5.82}$\Leftrightarrow $ Eq. \meqref{eq:5.89}$\Leftrightarrow $Eq. \meqref{eq:5.810}.
\item
	Eq. \meqref{eq:5.83}$\Leftrightarrow $Eq. \meqref{eq:5.84}$\Leftrightarrow $ Eq. \meqref{eq:5.87}$\Leftrightarrow $Eq. \meqref{eq:5.88}.
\item Eq. \meqref{eq:5.85}$\Leftrightarrow $Eq. \meqref{eq:5.812}.
\item	Eq. \meqref{eq:5.86}$\Leftrightarrow $Eq. \meqref{eq:5.811}.
\end{enumerate}
The same argument as in the proof of~\cite[Theorem~2.2.3]{Bai1} shows that Eq. \meqref{eq:5.81} (resp. \meqref{eq:5.83},  \meqref{eq:5.85} and \meqref{eq:5.86}) is equivalent to Eq.~\meqref{eq:nvb1} (resp. \meqref{eq:nvb2}, \meqref{eq:nvb3} and \meqref{eq:nvb4}). Furthermore, by Lemma \mref{lem:nva-equi}, $(A, \Delta_{\prec_{A}},\Delta_{\su_{A}})$ is a \nvca.
 Therefore, $(A,\prec_{A},\su_{A},\Delta_{\prec_{A}},\Delta_{\su_{A}})$ is a \nvbia.
\smallskip

\noindent
((b)$\Leftarrow$(c))
Since $(A, \prec_{A},\su_{A} )$ is a noncommutative Novikov algebra,
$(A^*,-R^*_{\ob_A}, L^*_{\succ_A}, R^*_{\prec_A}, -L^*_{\ob_A})$ is  a representation of $(A, \prec_{A},\su_{A} )$ by Example \mref{ex:ap}.  Similarly,  $(A,-R^*_{\ob_{A^*}},L^*_{\succ_{A^*}}, R^*_{\prec_{A^*}}, -L^*_{\ob_{A^*}})$ is  a representation of $(A^*, \prec_{A^*},\su_{A^*} )$. Eqs.~\meqref{eq:5.81}--\meqref{eq:5.812} follow from the above proof.
Then by Lemma~\mref{lem:5.8},
$(A,A^*,-R^*_{\ob_A}, L^*_{\succ_A}, R^*_{\prec_A}, -L^*_{\ob_A},
			-R^*_{\ob_{A^*}}, L^*_{\succ_{A^*}}, R^*_{\prec_{A^*}}, -L^*_{\ob_{A^*}})$ is a matched pair of \nvas.
\end{proof}

\vse

\section{Noncommutative Novikov bialgebras, Novikov Yang-Baxter equations and $\mathcal{O}$-operators}
\mlabel{sec:NYBE}
In this section, we give the notion of the Novikov Yang-Baxter equation (NYBE) in a noncommutative Novikov algebra. Then an antisymmetric solution of the NYBE leads to a triangular noncommutative Novikov bialgebra. As an operator form of the NYBE,  an $\mathcal{O}$-operator  on a \nva is introduced to construct  antisymmetric solutions of the NYBE. 
\subsection{Triangular \nvbias and NYBEs}
\begin{defi}
Let $(A,\prec,\succ)$ be a \nva and $r=\sum\limits_{i} u_{i}\otimes v_{i}\in A\otimes A$.
Define
\begin{equation*}
r_{12}\succ r_{13}:=\sum_{i,j}u_i\succ u_j\otimes v_i\otimes
v_j,\;\;r_{13}\prec r_{23}:=\sum_{i,j}u_i\otimes u_j\otimes
v_i\prec v_j,\;\;r_{23}\ob r_{12}:=\sum_{ij} u_i\otimes u_j\ob v_i\otimes v_j.
\end{equation*}
The equation
$$E(r):=r_{12}\succ r_{13}+r_{13}\prec r_{23}+r_{23}\ob r_{12}=0$$
is called the {\bf Novikov Yang-Baxter equation} (or the {\bf NYBE} in short) in $(A,\prec,\succ)$.
If $E(r)=0$, then we say $r$ is a {\bf solution} of the NYBE in $(A,\prec,\succ)$.
\end{defi}
Thus, one can also give an explicit presentation of the NYBE as follows:
\begin{equation}
E(r)=\sum_{i,j}u_{i}\succ u_{j}\otimes v_{i}\otimes v_{j}+u_{i}\otimes u_{j}\otimes v_{i}\prec v_{j}+u_{i}\otimes u_{j}\ob v_{i}\otimes v_{j}
=0.
\mlabel{eq:nybe}
\end{equation}
\begin{rmk} When the commutativity $x\prec y=y\succ x$ holds, $(A,\prec)$ is a   Novikov algebra. In this case, the NYBE becomes
$$E(r):= r_{13}\prec r_{12} +r_{13}\prec r_{23}+r_{23}\ob r_{12}=0,$$
which is the same as the NYBE in the  Novikov algebra given by \cite[Definition~3.24]{HBG23}.
\end{rmk}

Let $S_3$ be the symmetric group of order 3. We adopt the classical cycle notation with parentheses, such as, $(123)\in S_3$.  The linear map $\sigma_\omega: A\ot A\ot A\to A\ot A\ot A$ induced by $\omega\in S_3$ is given by $\sigma_\omega(a_1\ot a_2\ot a_3):= a_{\omega^{-1}(1)}\ot a_{\omega^{-1}(2)}\ot a_{\omega^{-1}(3)}$ for all $a_1, a_2, a_3\in A$.
\begin{pro}
Let $(A,\prec,\succ)$ be a \nva and $r\in A\otimes A$. If $r$ is an antisymmetric solution of the NYBE in $(A,\prec,\succ)$, then $(A,\prec, \succ,\Delta_{\prec},\Delta_{\succ})$ is a \nvbia, where the comultiplications $\Delta_{\prec},\Delta_{\succ}$ are given by
\begin{equation}\mlabel{eq:comultiplication}
		\Delta_{\prec}(x)=(\mathrm{id}\otimes L_{\ob}(x)+{R}_{\su}(x)\otimes\mathrm{id})(-r),\;
		\Delta_{\succ}(x)=(\mathrm{id}\otimes L_{\prec}(x)+{R}_{\ob}(x)\otimes\mathrm{id})r,\quad x\in A.
\end{equation}
In this case, $(A,\prec, \succ,\Delta_{\prec},\Delta_{\succ})$ is called a {\bf triangular \nvbia}.
\mlabel{pro:antieybe}
\end{pro}
\begin{proof}
By Eq.~\meqref{eq:comultiplication}, we have
\begin{equation}\mlabel{eq:hx}
\Delta_\ob(x)=(\id\ot L_\succ(x)-R_\prec(x)\ot\id)(-r).
\end{equation}

We first prove that $(A,\Delta_{\prec},\Delta_{\succ})$ is a \nvca by checking Eqs.~\meqref{eq:exca1} and ~\meqref{eq:exca2}.
We will only verify Eq.~\meqref{eq:exca1} --- the verification for Eq.~\meqref{eq:exca2} is similar. For all $x \in A$, we have
\begin{eqnarray*}
	(\Delta_{\succ}\otimes \id)\Delta_{\prec}(x)&=&(\Delta_{\succ}\otimes \id)(\id \otimes L_{\ob}(x)+ R_{\succ}(x) \otimes \id )(-\sum\limits_{i} u_{i} \otimes v_{i})\\
	&=&(\Delta_{\succ}\otimes \id)(\sum\limits_{i} -u_{i} \otimes x \ob v_{i}-u_{i} \succ x \otimes v_{i})\\
	&=&\sum\limits_{i} -\Delta_{\succ} (u_{i}) \otimes x \ob v_{i}- \Delta_{\succ}(u_{i} \succ x) \otimes v_{i}\\
	&=&\sum\limits_{i}-(\id \otimes L_{\prec}(u_{i})+R_{\ob}(u_{i}) \otimes \id)(\sum\limits_{j} u_{j} \otimes v_{j})\otimes x \ob v_{i} \\
	&&-(\id \otimes L_{\prec}(u_{i} \succ x)+R_{\ob}(u_{i}\succ x)\otimes \id)(\sum\limits_{j} u_{j} \otimes v_{j}) \otimes v_{i}\\
	&=&\sum\limits_{i,j}-u_{i}\otimes u_{j}\otimes x \ob (v_{i} \succ v_{j})-u_{i} \otimes u_{j} \succ (x \prec v_{i}) \otimes v_{j} - u_{j} \ob (u_{i}\succ x) \otimes v_{j} \otimes v_{i}\\
&&\text{(by  acting $\sigma_{(132)}$ on the NYBE, the antisymmetry of $r$ and Eq.~\meqref{eq:nv1})}.
\end{eqnarray*}

On the other hand,
\begin{eqnarray*}
	(\id \otimes \Delta_{\prec}) \Delta_{\succ}(x)	&=&(\id \otimes \Delta_{\prec})(\id \otimes L_{\prec}(x)+ R_{\ob}(x) \otimes \id )(\sum\limits_{i} u_{i} \otimes v_{i})\\
	&=&(\id \otimes \Delta_{\prec})(\sum\limits_{i} u_{i} \otimes x \prec v_{i}+u_{i} \ob x \otimes v_{i})\\
	&=&\sum\limits_{i,j}u_{i} \otimes \Delta_{\prec}(x\prec v_{i})+u_{i}\ob x \otimes \Delta_{\prec}(v_{i})\\
	&=&\sum\limits_{i,j} u_{i} \otimes (\id \otimes L_{\ob}(x\prec v_{i})+R_{\succ}(x\prec v_{i})\otimes \id)(-u_{j}\otimes v_{j})\\
	&&+u_{i} \ob x \otimes (\id \otimes L_{\ob}(v_{i}) + R_{\succ}(v_{i})\otimes \id)(-u_{j}\otimes v_{j})\\
	&=&\sum\limits_{i,j}-u_{i}\otimes u_{j}\otimes x\ob (v_{i} \succ v_{j}) -u_{i}\otimes u_{j} \succ (x\prec v_{i})\otimes v_{j}-(v_{i} \prec v_{j})\ob x\ot u_{i}\otimes u_{j}\\
	&&\text{(by acting $\sigma_{(123)}$ on the NYBE, Eq.~\meqref{eq:nv2} and the antisymmetry of $r$)}\\
	&=&\sum\limits_{i,j} -u_{i}\otimes u_{j}\otimes x\ob (v_{i} \succ v_{j}) -u_{i}\otimes u_{j} \succ (x\prec v_{i})\otimes v_{j}- u_{j} \ob (u_{i} \succ x) \otimes v_{j} \otimes v_{i}\\
	&&\text{(by Eq.~\meqref{eq:nv2} and the antisymmetry of $r$)}.
\end{eqnarray*}
Thus, Eq.~\meqref{eq:exca1} holds.

We next prove  Eqs.~\meqref{eq:nvb1}--\meqref{eq:nvb4}.
We only verify Eq.~\meqref{eq:nvb1}, since the others are proved similarly.
For all $x,y \in A$, we have
\begin{eqnarray*}
\Delta_{\prec}(x\prec y)&=&(\id \ot L_{\ob}(x\prec y)+R_{\su}(x\prec y)\ot \id)(-\sum\limits_{i} u_{i} \otimes v_{i})\\
&=&\sum\limits_{i} -u_{i} \otimes(x\prec y)\ob v_i-u_{i} \su (x\prec y)\ot v_i.\\
&&(R_{\prec}(y)\ot \id)\Delta_{\prec}(x)+(\id\ot L_{\ob}(x))\Delta_{\ob}(y)\\
&=&(R_{\prec}(y)\ot \id)(\id \ot L_{\ob}(x)+R_{\su}(x) \ot \id)(-\sum\limits_{i} u_{i} \otimes v_{i})\\
&&+(\id \ot L_{\ob}(x))(\id \ot L_{\su}(y)-R_{\prec}(y)\ot \id)(-\sum\limits_{i} u_{i} \otimes v_{i})\\
&=&\sum\limits_{i}-u_i\prec y \ot x\ob v_i-(u_i\su x)\prec y \ot v_i -u_i\ot x\ob(y\su v_i)+u_i \prec y \ot x\ob v_i\\
&=&\sum\limits_{i}-u_i\su (x\prec y)\ot v_i-u_i\ot (x\prec y)\ob v_i.
\end{eqnarray*}
Thus, Eq.~\meqref{eq:nvb1} holds. This completes the proof. 
\end{proof}

\subsection{$\mathcal{O}$-operators of  \nvas and NYBEs}
We now give the notion of an $\mathcal{O}$-operator of  \nvas.
\begin{defi}\mlabel{defi:ope}
Let $(A,\prec,\succ)$ be a \nva and  $(V,\ell_\prec, r_\prec, \ell_\su, r_\su)$ be a representation. A linear map
$T: V\to A$ is called an {\bf $\mathcal{O}$-operator} of $(A,\prec,\succ)$  associated  to $(V,\ell_\prec, r_\prec, \ell_\su, r_\su)$, if $T$ satisfies
\begin{equation}
T(u)\ast T(v)=T(\ell_\ast(T(u))v+r_\ast(T(v))u),\quad \ast\in \{\prec,\su\},\;  u,v\in V.
\mlabel{eq:ope}
\end{equation}
\end{defi}

\begin{pro}\mlabel{pro:oelda}
Let $(A,\prec,\succ)$ be a \nva and $r\in A\otimes A$ be antisymmetric.
Then the following statements are equivalent:
\begin{enumerate}
\item \mlabel{it:eybe}
$r$ is a solution of the NYBE in $(A,\prec,\succ)$.
\item \mlabel{it:ope}
${r}$ is an $\mathcal{O}$-operator of $(A,\prec,\succ)$ associated to
$(A^*, -{R}^{*}_{\ob}, {L}^{*}_{\succ}, {R}^{*}_{\prec}, -{L}^{*}_{\ob})$.
\end{enumerate}
\end{pro}
\begin{proof}
 By Example~\mref{ex:ap}, $(A^*, -{R}^{*}_{\ob}, {L}^{*}_{\succ}, {R}^{*}_{\prec}, -{L}^{*}_{\ob})$ is a representation of \nva. For all $a^*,b^*,c^* \in A^*$, we have
$$ \langle r(a^*) \succ r(b^*),c^* \rangle =\langle (\sum\limits_{i}\langle a^*,u_{i}\rangle v_{i}) \succ (\sum\limits_{j}\langle b^*,u_{j}\rangle v_{j}),c^* \rangle
=\langle \sum\limits_{i,j} u_{i} \otimes u_{j} \otimes v_{i} \succ v_{j}, a^*\otimes b^*\otimes c^* \rangle.
$$
\begin{eqnarray*}
\langle r(R_{\prec}^*(r(a^*))b^*),c^* \rangle &=& \langle \sum\limits_{i} \langle R_{\prec}^*(r(a^*))b^*, u_{i} \rangle v_{i} ,c^* \rangle\\
&=& \sum\limits_{i} \langle b^*, u_{i} \prec r(a^*) \rangle \langle v_{i},c^* \rangle\\
&=& \sum\limits_{i} \langle b^*, u_{i} \prec (\sum\limits_{j} \langle a^*,u_{j} \rangle v_{j}) \rangle \langle v_{i},c^* \rangle\\
&=& \sum\limits_{i,j} \langle a^*,u_{j} \rangle \langle b^*,u_{i}
\prec v_{j} \rangle \langle c^*, v_{i} \rangle\\
&=& \langle \sum\limits_{i,j} u_{j} \otimes u_{i} \prec v_{j} \otimes v_{i} ,a^*\otimes b^*\otimes c^* \rangle.
\end{eqnarray*}
\begin{eqnarray*}
\langle r(-L_{\ob}^*(r(b^*))a^*),c^* \rangle &=& -\langle \sum\limits_{i} \langle L_{\ob}^*(r(b^*))a^*, u_{i} \rangle v_{i} ,c^* \rangle\\
&=& -\sum\limits_{i} \langle a^*,  r(b^*) \ob u_{i} \rangle \langle v_{i},c^* \rangle\\
&=& -\sum\limits_{i} \langle a^*,(\sum\limits_{j} \langle b^*,u_{j} \rangle v_{j}) \ob u_{i} \rangle \langle v_{i},c^* \rangle\\
&=& -\sum\limits_{i,j} \langle b^*,u_{j} \rangle \langle a^*,v_{j}
\ob u_{i} \rangle \langle c^*, v_{i} \rangle\\
&=& -\langle \sum\limits_{i,j} v_{j} \ob u_{i} \otimes u_{j}  \otimes v_{i} ,a^*\otimes b^*\otimes c^* \rangle.
\end{eqnarray*}
Thus,
 $r(a^*) \succ r(b^*)=r(R_{\prec}^*(r(a^*))b^*-L_{\ob}^*(r(b^*))a^* )$ if and only if
\begin{eqnarray*}
 \sum\limits_{i,j} u_{i} \otimes u_{j} \otimes v_{i} \succ v_{j}= \sum\limits_{i,j} u_{j} \otimes u_{i} \prec v_{j} \otimes v_{i}-v_{j} \ob u_{i} \otimes u_{j}  \otimes v_{i},
\end{eqnarray*}
 which, by the antisymmetry of $r$ and the action of $\sigma_{(123)}$, is equivalent to
$$ \sum\limits_{i,j} u_{i} \succ u_{j} \otimes v_{i} \otimes v_{j} +u_{i} \otimes u_{j} \otimes v_{i} \prec v_{j} + u_{i} \otimes u_{j} \ob v_{i} \otimes v_{j}=0.$$
By the same argument, we prove that
$r(a^*) \prec r(b^*)=r(-R_{\ob}^*(r(a^*))b^*+L_{\succ}^*(r(b^*))a^* )$ if and only if
\begin{eqnarray*}
\sum\limits_{i,j} u_{i} \otimes u_{j} \otimes v_{i} \prec v_{j}= -\sum\limits_{i,j} u_{j} \otimes u_{i} \ob v_{j} \otimes v_{i}+v_{j} \succ u_{i} \otimes u_{j}  \otimes v_{i},
\end{eqnarray*}
which is equivalent to
$$
\sum\limits_{i,j} u_{i} \otimes u_{j} \otimes v_{i} \prec v_{j}+\sum\limits_{i,j} u_{i} \otimes u_{j} \ob v_{i} \otimes v_{j}+u_{i} \succ u_{j} \otimes v_{i}  \otimes v_{j}=0.$$
Now the proof is completed.
\end{proof}

\begin{thm}\mlabel{thm:5.8}Let  $A:=(A,\prec,\succ)$ be a \nva and $(V,\ell_\prec,r_\prec,\ell_\su,r_\su)$ be a representation of  $A$. Let $T:V\rightarrow A$ be a linear map identified as a $2$-tensor  $r_T\in V^*\ot A\subseteq (A\ltimes_{\ell^*_{\su},-\ell^*_\ob}^{-r^*_\ob,r^*_\prec}V^*)\otimes
(A\ltimes_{\ell^*_{\su},-\ell^*_\ob}^{-r^*_\ob,r^*_\prec}V^*).$  Then $r:=r_T-\sigma(r_T)$ is an antisymmetric solution of  the NYBE in the semi-direct product  \nva $A\ltimes_{\ell^*_{\su},-\ell^*_\ob}^{-r^*_\ob,r^*_\prec}V^*$ defined by Proposition~\mref{pro:ld-rep}, if and only if $T$ is an $\mathcal{O}$-operator of  $(A,\prec,\succ)$ associated to $(V,\ell_\prec,r_\prec,\ell_\su,r_\su)$.
\end{thm}
\begin{proof}
Let $\{e_{1},\cdots,e_{n}\}$ be a basis of $V$ and let $\{e^*_{1},\cdots,e^*_{n}\}$ be its dual basis. Since $T$ can be regarded as an element in $V^*\ot A$, we write $r_T=\sum\limits_{i=1}^{n} e_{i}^* \ot T(e_{i})$. Thus, $r=r_T-\sigma(r_T)=\sum\limits_{i=1}^{n} (e_{i}^* \ot T(e_{i}) -  T(e_{i}) \ot e_{i}^*)$. So we have
\begin{eqnarray*}
r_{12} \succ' r_{13} &=&
\sum\limits_{i,j} e_{i}^* \succ' e_{j}^* \ot T(e_{i}) \ot T(e_{j}) -  e_{i}^*\succ' T(e_{j}) \ot T(e_{i}) \ot e_{j}^* \\
&&- T(e_{i})  \succ' e_{j}^* \ot e_{i}^* \ot T(e_{j}) +  T(e_{i}) \succ'T(e_{j}) \ot e_{i}^*\ot e_{j}^* \\
&=&
\sum\limits_{i,j} \ell_{\ob}^* (T(e_{j}))e_{i}^*   \ot T(e_{i}) \ot e_{j}^* -  r_{\prec}^*(T(e_{i}))  (e_{j}^*) \ot e_{i}^* \ot T(e_{j})  +T(e_{i}) \succ T(e_{j}) \ot e_{i}^* \ot e_{j}^*.\\
r_{13} \prec' r_{23} &=&
\sum\limits_{i,j} e_{i}^* \ot e_{j}^* \ot T(e_{i}) \prec'T(e_{j}) -  e_{i}^* \ot T(e_{j}) \ot   T(e_{i}) \prec' e_{j}^* \\
&&- T(e_{i}) \ot e_{j}^* \ot  e_{i}^* \prec'T(e_{j}) +  T(e_{i}) \ot T(e_{j})  \ot  e_{i}^*\prec' e_{j}^* \\
&=&\sum\limits_{i,j} e_{i}^* \ot e_{j}^* \ot T(e_{i}) \prec T(e_{j})  + e_{i}^* \ot T(e_{j}) \ot  r_{\ob}^*(T(e_{i}))  (e_{j}^*)    -T(e_{i})\ot e_{j}^*\ot \ell_{\succ}^* (T(e_{j})) e_{i}^* .\\
r_{23} \ob' r_{12} &=&
\sum\limits_{i,j} e_{i}^* \ot e_{j}^* \ob' T(e_{i})\ot T(e_{j})    - e_{i}^* \ot T(e_{j}) \ob' T(e_{i})  \ot   e_{j}^* \\
&&- T(e_{i}) \ot e_{j}^* \ob'  e_{i}^*\ot T(e_{j}) +  T(e_{i}) \ot T(e_{j})  \ob'  e_{i}^*\ot e_{j}^* \\
&=&\sum\limits_{i,j}  e_{i}^*  \ot  (\ell_{\succ}^*-\ell_{\ob}^*)(T(e_{i})) e_{j}^* \ot T(e_{i})  -e_{i}^*\ot T(e_{j}) \ob T(e_{i})\ot e_{j}^* \\
&&+ T(e_{i}) \ot (r_{\prec}^*-r_{\ob}^*) (T(e_{j})) e_{i}^* \ot e_{j}^*.
\end{eqnarray*}
Then  $r$ is an antisymmetric solution of the NYBE in  $(A\ltimes_{\ell^*_{\su},-\ell^*_\ob}^{-r^*_\ob,r^*_\prec}V^*,\succ',\prec')$, if and only if
\begin{eqnarray*}
&&\sum_{i,j} e_{i}^* \ot e_{j}^* \ot T(e_{i}) \prec T(e_{j})  + e_{i}^* \ot T(e_{j}) \ot  r_{\ob}^*(T(e_{i}))  (e_{j}^*)    -T(e_{i})\ot e_{j}^*\ot \ell_{\succ}^* (T(e_{j})) e_{i}^*\\
 &&+e_{i}^*  \ot  (\ell_{\succ}^*-\ell_{\ob}^*)(T(e_{i})) e_{j}^* \ot T(e_{i})  -e_{i}^*\ot T(e_{j}) \ob T(e_{i})\ot e_{j}^* + T(e_{i}) \ot (r_{\prec}^*-r_{\ob}^*) (T(e_{j})) e_{i}^* \ot e_{j}^*\\
&&+\ell_{\ob}^* (T(e_{j}))e_{i}^*   \ot T(e_{i}) \ot e_{j}^* -  r_{\prec}^*(T(e_{i}))  (e_{j}^*) \ot e_{i}^* \ot T(e_{j})  +T(e_{i}) \succ T(e_{j}) \ot e_{i}^* \ot e_{j}^* =0.
\end{eqnarray*}
This is equivalent to
\begin{eqnarray}
&&\sum\limits_{i,j}
e_{i}^* \ot e_{j}^* \ot T(e_{i}) \prec T(e_{j})-r_{\prec}^*(T(e_{i}))e_{j}^*\ot e_{i}^*\ot T(e_{j})
-e_{i}^*  \ot  \ell_{\prec}^*(T(e_{i})) e_{j}^* \ot T(e_{i})=0,
\mlabel{eq:equ1}\\
&&\sum\limits_{i,j}
T(e_{i}) \succ T(e_{j}) \ot e_{i}^* \ot e_{j}^*
-T(e_{i})\ot e_{j}^*\ot \ell_{\succ}^* (T(e_{j})) e_{i}^*
- T(e_{i}) \ot r_{\succ}^* (T(e_{j}))e_{i}^* \ot e_{j}^*=0,
\mlabel{eq:equ2}\\
&&\sum\limits_{i,j} \ell_{\ob}^* (T(e_{j}))e_{i}^*\ot T(e_{i})\ot e_{j}^*
+e_{i}^* \ot T(e_{j})\ot r_{\ob}^*(T(e_{i})) e_{j}^*
-e_{i}^*\ot T(e_{j}) \ob T(e_{i})\ot e_{j}^* =0.
\mlabel{eq:equ3}
\end{eqnarray}
Write
\begin{eqnarray*}
r_{\prec}^* (T(e_{i}))e_{j}^*=\sum\limits_{k} \langle r_{\prec}^* (T(e_{i}))e_{j}^*,e_{k} \rangle e_{k}^*,\quad
\ell_{\prec}^*(T(e_{i}))e_j^*=\sum\limits_{k}\langle \ell_{\prec}^*(T(e_{i}))e_{j}^*,e_{k} \rangle e_{k}^*.
\end{eqnarray*}
Then we obtain
\begin{eqnarray*}
	&& \sum\limits_{i,j}-r_{\prec}^*(T(e_{i}))e_{j}^* \ot e_i^* \ot T(e_j) - e_{i}^* \ot \ell_{\prec}^* (T(e_{i}))e_{j}^*\ot  T(e_{i}) \\
	&=& \sum\limits_{i,j,k}-\langle r_{\prec}^*(T(e_{i}))e_{j}^*,e_{k} \rangle e_k ^*\ot e_i^* \ot T(e_j) -e_{i}^*\ot \langle \ell^*_{\prec}(T(e_{i}))e_{j}^*,e_{k}\rangle e_k ^*\ot T(e_{j})  \\
	&=&\sum\limits_{i,j,k}-\langle e_{j}^*, r_{\prec}(T(e_{i}))e_{k} \rangle e_k ^*\ot e_i^* \ot T(e_j) -e_{i}^*\ot \langle e_{j}^* ,\ell_{\prec}(T(e_{i}))e_{k}\rangle e_k ^*\ot T(e_{j}) \\
	&=&\sum\limits_{i,k}- e_{k}^*\ot e_{i}^*\ot T(\sum\limits_{j}\langle e_{j}^*,r_{\prec}(T(e_{i}))e_k \rangle e_j)- e_{i}^*\ot e_{k}^* \ot  T(\sum\limits_{j}\langle e_{j}^*,\ell_{\prec}(T(e_{i}))e_k \rangle e_j) \\
	&=&-\sum\limits_{i,k}e_{k}^*\ot e_{i}^* \ot  T(r_{\prec}(T(e_{i}))e_k )-\sum\limits_{i,k} e_{i}^*\ot e_{k}^* \ot
	T(\ell_{\prec}(T(e_{i}))e_k )\\
	&=&-\sum\limits_{i,j}e_{i}^*\ot e_{j}^* \ot  T(r_{\prec}(T(e_{j}))e_i)-\sum\limits_{i,j} e_{i}^*\ot e_{j}^* \ot
	T(\ell_{\prec}(T(e_{i}))e_j)\\
	&=&\sum\limits_{i,j} e_{i}^*\ot e_{j}^* \ot
	T(-\ell_{\prec}(T(e_{i}))e_j- r_{\prec}(T(e_{j}))e_i).
\end{eqnarray*}
Then Eq.~\meqref{eq:equ1} holds if and only if
\begin{eqnarray*}
	T(e_{i}) \prec T(e_{j})=T(\ell_{\prec}(T(e_{i}))e_{j}
	+r_{\prec}(T(e_{j}))e_{i}).
\end{eqnarray*}
Analogously,  we can show that Eq.~\meqref{eq:equ2} holds if and only if
\begin{eqnarray*}
	T(e_{i}) \succ T(e_{j})= T(\ell_{\succ}(T(e_{i}))e_{j}
	+r_{\succ}(T(e_{j}))e_{i}).
\end{eqnarray*}
By $\ob=\prec+\succ$, we obtain
\begin{eqnarray*}
T(e_{i}) \ob T(e_{j})=T(\ell_{\ob}(T(e_{i}))e_{j}
+r_{\ob}(T(e_{j}))e_{i}),
\end{eqnarray*}
which is equivalent to Eq.~\meqref{eq:equ3}.
\end{proof}

\section{Noncommutative pre-Novikov algebras, and quasi-Frobenius \nvas}
\mlabel{sec:quasiFor}
In this section,  the notion of a \pnva is introduced. Then we show that an $\mathcal{O}$-operator of \nvas produces
 a \pnva. On the other hand, a \pnva naturally gives rise to an $\mathcal{O}$-operator on the associated \nva. Moreover, the concept of a quasi-Frobenius noncommutative Novikov algebra is given in order to
  characterize  antisymmetric nondegenerate solutions of the NYBE.

\subsection{Noncommutative pre-Novikov algebras and $\mathcal{O}$-operators}
Pre-Novikov algebras, arising from the operadic splitting of Novikov algebra operads, were investigated in ~\cite{HBG23,KMS24,LH24}. We establish their noncommutative counterparts.
\begin{defi}
A {\bf \pnva} is a quintuple $(A,\lhd_1, \lhd_2,\rhd_1, \rhd_2)$, where $A$ is a vector space and $\lhd_1, \lhd_2,\rhd_1, \rhd_2: A\ot A\to A$ are binary operations satisfying the following axioms, for all $x, y, z\in A$:  
\begin{align}
(x\su y)\lhd_1 z &= x \rhd_1(y\lhd_1 z),\mlabel{eq:pnva1}\\
(x\rhd_1 y)\lhd_2 z &= x \rhd_1(y\lhd_2 z),\mlabel{eq:pnva2}\\
(x\rhd_2 y)\lhd_2 z &= x \rhd_2(y\prec z),\mlabel{eq:pnva3}\\
(x\prec y)\lhd_1 z +(x\prec y)\rhd_1 z &=x \lhd_1(y\rhd_1 z)+ x \rhd_1(y\rhd_1 z),\mlabel{eq:pnva4}\\
(x\lhd_1 y)\lhd_2 z +(x\lhd_1 y)\rhd_2 z &=x \lhd_1(y\rhd_2 z)+ x \rhd_1(y\rhd_2 z),\mlabel{eq:pnva5}\\
(x\lhd_2 y)\lhd_2 z +(x\lhd_2 y)\rhd_2 z &=x \lhd_2(y\su z)+ x \rhd_2(y\su z),\mlabel{eq:pnva6}
\end{align}
where $\prec:=\lhd_1+\lhd_2$ , $\su:=\rhd_1+\rhd_2$.
\end{defi}

\begin{pro}
	Let $(A,\lhd_1, \lhd_2,\rhd_1, \rhd_2)$ be a \pnva. The binary operations
\begin{align}
	x\prec y&:=x\lhd_1 y+ x\lhd_2 y, \mlabel{eq:pnvab1}\\
	x\succ y&:=x\rhd_1 y+ x\rhd_2 y,\quad x,y\in A.\mlabel{eq:pnvab2}
\end{align}
	defines a \nva, which is called the {\bf associated \nva} of $(A,\lhd_1, \lhd_2,\rhd_1, \rhd_2)$. Furthermore,
    $(A,L_{\lhd_1},R_{\lhd_2},L_{\rhd_1}, R_{\rhd_2})$ is a representation of $(A,\prec,\su)$. Conversely, let A be a vector space equipped with binary operations $\lhd_1, \lhd_2,\rhd_1, \rhd_2$. If $(A,\prec,\su)$ defined by Eqs.~\meqref{eq:pnvab1} and~\meqref{eq:pnvab2} is a \nva, and $(A,L_{\lhd_1},R_{\lhd_2},L_{\rhd_1}, R_{\rhd_2})$ is a representation of $(A,\prec,\su)$, then $(A,\lhd_1,\lhd_2,\rhd_1,\rhd_2)$ is a \pnva.
\end{pro}
\begin{proof}
If $(A,\lhd_1, \lhd_2,\rhd_1, \rhd_2)$ is a \pnva, then Eqs.~\meqref{eq:pnva1}--\meqref{eq:pnva6} hold.  Thus, Eq.~\meqref{eq:nv1} follows from Eqs.~\meqref{eq:pnva1}--\meqref{eq:pnva3}. Similarly, Eq.~\meqref{eq:nv2} follows from Eqs.~\meqref{eq:pnva4}--\meqref{eq:pnva6}. Moreover, by Definition~\mref{defi:5.3}, $(A,L_{\lhd_1},R_{\lhd_2},L_{\rhd_1}, R_{\rhd_2})$ satisfies Eqs.~\meqref{eq:r1}--\meqref{eq:r6}, and is a representation of $(A,\prec,\su)$. 

Conversely, if $(A,\prec,\su)$  is a \nva and $(A,L_{\lhd_1},R_{\lhd_2},L_{\rhd_1}, R_{\rhd_2})$ is a representation of $(A,\prec,\su)$,  Eqs.~\meqref{eq:pnva1}--\meqref{eq:pnva6} follow from Eqs.~\meqref{eq:r1}--\meqref{eq:r6}.
\end{proof}

\begin{pro} \mlabel{pro:pnva}
Let $(A,\prec,\su)$ be a \nva, $(V,\ell_\prec,r_\prec,\ell_\su,r_\su)$ be a representation of $(A,\prec,\su)$ and $T:V\rightarrow A$ be an $\calo$-operator of $(A,\prec,\su)$ associated to $(V,\ell_\prec,r_\prec,\ell_\su,r_\su)$. Then there exsits a \pnva structure on $V$ defined by
\begin{align}
u \lhd_1 v&=\ell_\prec(T(u))v,\;
u \lhd_2 v=r_\prec(T(v))u,\mlabel{eq:op1}\\
u \rhd_1 v&=\ell_\su(T(u))v,\;
u \rhd_2 v=r_\su(T(v))u,\quad u,v\in V.\mlabel{eq:op2}
\end{align}
Define
\begin{align}
u \lhd_1 v&=\ell_\prec(T(u))v,\;
u \lhd_2 v=r_\prec(T(v))u,\mlabel{eq:op1}\\
u \rhd_1 v&=\ell_\su(T(u))v,\;
u \rhd_2 v=r_\su(T(v))u,\quad u,v\in V.\mlabel{eq:op2}
\end{align}
If $\prec$ and $\succ$ satisfy Eqs.~\meqref{eq:pnvab1} and~\meqref{eq:pnvab2}, then $(A,\lhd_1,\lhd_2,\rhd_1,\rhd_2)$ is a \pnva.
\end{pro}
\begin{proof}
	For all $u,v,w \in V$, we have
\begin{align*}
(u\su v)\lhd_1 w &= \ell_\prec(T(u \rhd_1 v+u \rhd_2 v))w
= \ell_\prec(T(\ell_\su(T(u)) v+r_\su(T(v))u))w\\
&=\ell_\prec(T(u)\su T(v))w
=\ell_\su(T(u))\ell_\prec(T(v))w\quad(\text{by Eq.~\meqref{eq:r1}})\\
&=u \rhd_1 (\ell_\prec(T(v))w)
=u \rhd_1(v\lhd_1 w).
\end{align*}
Then Eq.~\meqref{eq:pnva1} follows. Eqs.~\meqref{eq:pnva2}--\meqref{eq:pnva6} are verified by the same method.
\end{proof}

Conversely, we have
\begin{pro}\mlabel{pro:id} Let $(A,\lhd_1, \lhd_2,\rhd_1, \rhd_2)$ be a \pnva  and $(A,\prec,\su)$ be the associated \nva.
 Then the identity map $\id$ on $A$ is an $\mathcal{O}$-operator of $(A,\prec,\su)$ associated to the induced representation
$(A,L_{\lhd_1},R_{\lhd_2},L_{\rhd_1}, R_{\rhd_2})$.
\end{pro}
\begin{proof} By Eqs.~\meqref{eq:pnvab1}--\meqref{eq:pnvab2}, we have 
$$x\prec y=x\lhd_1 y+x\lhd_2 y=L_{\lhd_1}(x)y+R_{\lhd_2}(y)x,\;
x\su y=x\rhd_1 y+x\rhd_2 y=L_{\rhd_1}(x)y+R_{\rhd_2}(y)x.$$
By Definition~\mref{defi:ope}, $\id$ is an $\mathcal{O}$-operator of $(A,\prec,\su)$ associated to $(A,L_{\lhd_1},R_{\lhd_2},L_{\rhd_1}, R_{\rhd_2})$.
\end{proof}

\subsection{Quasi-Frobenius \nvas and NYBEs}
By ~\cite[Corollary]{D1983}, there is an intimate relation between antisymmetric solutions of the CYBE and
 quasi-Frobenius Lie algebras. That is, for a given finite-dimensional Lie algebra $\frakg$, $r\in \frakg\ot \frakg$ is an antisymmetric nondegenerate solution of the CYBE in $\frakg$ if and only if $\frakg$ is a quasi-Frobenius Lie algebra (see also \cite[Theorem 1.1]{BFS97}).  This result has an analog for Novikov algebras by~\cite[Proposition 3.26]{HBG23}.   In order to extend these results to the context of \nvas, we present the following notion.
\begin{defi}
Let $(A,\prec,\su)$ be a \nva. If there is an antisymmetric nondegenerate bilinear form $\frakB$ on A satisfying
\begin{eqnarray}
\frakB(x\prec y,z)+\frakB(y\su z,x)-\frakB(z\ob x,y)=0,\quad x,y\in A,\mlabel{eq:109}
\end{eqnarray}
then $(A,\prec,\su,\frakB)$ is called a {\bf quasi-Frobenius \nva}.
\end{defi}
For finite dimensional vector spaces $V$ and $W$, a tensor $r=\sum_{i}u_i\ot v_i\in V\otimes W$  can be  viewed as a linear map $r^\sharp:V^{*}\rightarrow W$ by
\begin{equation}\mlabel{eq:tensor}
	{r^\sharp}(a^{*})=\sum_i \sqmon{a^*, u_i}v_i,\quad a^{*}\in V^{*}.
\end{equation}
Thus an element $r=\sum_{i}u_i\ot v_i\in A\ot A$ induces a linear map $r^\sharp: A^*\to A$
by  
\begin{equation}\mlabel{eq:as}
\langle r^\sharp(a^*), b^*\rangle:=\langle r, a^*\ot b^*\rangle.
\end{equation}
Then $r$ is said to be {\bf nondegenerate} if $r^\sharp:A^*\to A$ is an isomorphism. 

\begin{pro}\mlabel{pro:4.6}
Let $(A,\prec,\su)$ be a \nva. Then $r\in A\ot A$ is an
antisymmetric nondegenerate solution of the NYBE in $(A,\prec,\su)$ if and only if $(A,\prec,\su,\frakB)$ is a quasi-Frobenius noncommutative Novikov algebra, where the bilinear form $\frakB$ on $A$ is defined by
$$\frakB(x,y):=\langle (r^\sharp)^{-1}(x),y\rangle,\quad x,y\in A.$$
\end{pro}
\begin{proof}For all $a^*,b^*,c^*\in A^*$, we set $r^\sharp(a^*)=x, r^\sharp(b^*)=y, r^\sharp(c^*)=z$ for some $x,y,z\in A$. Then by Eq.~\meqref{eq:as},
$$\frakB(x,y)=\langle (r^\sharp)^{-1}(x),y\rangle=\langle a^*,r^\sharp(b^*)\rangle=\langle r, b^*\ot a^*\rangle.$$
$$\frakB(y,x)=\langle (r^\sharp)^{-1}(y),x\rangle=\langle b^*,r^\sharp(a^*)\rangle=\langle r, a^*\ot b^*\rangle=\langle \sigma(r), b^*\ot a^*\rangle.$$
Thus $\frakB$ is antisymmetric if and only if $r$ is antisymmetric.

Furthermore, we have 
\begin{align*}
\frakB(x\prec y,z)&=-\frakB(z,x\prec y)=-\langle (r^\sharp)^{-1}(z), L_\prec(x)y\rangle\\
&=-\langle L^*_\prec(x)(c^*), y\rangle=-\langle L^*_\prec(x)(c^*), r^\sharp(b^*)\rangle\\
&=-\langle r, b^* \ot  L^*_\prec(x)(c^*)\rangle=-\langle(\id\ot L_\prec(x)) r, b^* \ot  c^*\rangle\quad(\text{by Eq.~\meqref{eq:as}})\\
&=-\langle\sum_i u_i\ot (x \prec v_i), b^* \ot  c^*\rangle
=-\langle\sum_i u_i\ot (r^\sharp(a^*) \prec v_i), b^*\ot  c^*\rangle\\
 &=-\langle\sum_i u_i\ot (\sum_j a^*(u_j)v_j\prec v_i), b^* \ot  c^*\rangle\quad(\text{by Eq.~\meqref{eq:tensor}})\\
&=-\langle\sum_{i,j} u_j\ot u_i\ot ( v_j\prec v_i), a^*\ot b^* \ot  c^*\rangle=-\langle\sum_{i,j} u_i\ot u_j\ot ( v_i\prec v_j), a^*\ot b^* \ot  c^*\rangle.
\end{align*}
By the same argument, we get
\begin{align*}
\frakB(y\su z,x)&=-\langle \sum_{i,j}u_i\succ u_j\ot v_i\ot v_j,a^*\ot b^*\ot c^*\rangle.\\
-\frakB(z\ob x,y)&=-\langle \sum_{i,j}u_i\ot u_j\ob v_i\ot v_j,a^*\ot b^*\ot c^*\rangle.
\end{align*}
Thus, $r$ is a solution of the NYBE in $(A,\prec,\su)$ if and only if $\frakB$ satisfies Eq.~\meqref{eq:109}.
\end{proof}

The notion of a dendriform algebra was introduced by Loday~\mcite{Lo} with motivation from algebraic $K$-theory. We now provide a construction of  \pnvas  from differential dendriform algebras.
\begin{ex}
A {\bf dendriform algebra} is a vector space $A$ equipped with binary operations $\prec_d$ and $\succ_d$  satisfying the following relations:
	\begin{align}
		(x \prec_d y) \prec_d z &= x \prec_d (y\prec_d z +y \succ_d z),\mlabel{eq:dd1}\\
		(x \succ_d y ) \prec_d z&= x \succ_d (y\prec_d z),\mlabel{eq:dd2}  \\
		(x \prec_d y +x\succ_d y)\succ_d z &= x \succ_d (y\succ_d z), \quad x, y, z\in D.
		\mlabel{eq:dd3}
	\end{align}
Here $\prec_d$ and $\succ_d$ are used to distinguish from $\prec$ and $\succ$ in a noncommutative Novikov algebra. 
A {\bf differential dendriform algebra} is a dendriform algebra $(A,\prec_d,\succ_d)$ equipped with a linear map $D$ on $A$ such that  
$$D(x\prec_d y)= D(x)\prec_d y+ x\prec_d D(y)\quad\text{and}\quad D(x\succ_d y)= D(x)\succ_d y+ x\succ_d D(y),\quad x,y\in A.$$
Then we define binary operations  $\lhd_1,\lhd_2$ and $\rhd_1,\rhd_2: A\otimes
A\rightarrow A$ by
$$x\lhd_1 y\coloneqq x\succ_d D(y),\;\;x\rhd_1 y\coloneqq D(x)\succ_d y,\;x\lhd_2 y \coloneqq x\prec_d D(y),\;\;x\rhd_2 y\coloneqq D(x)\prec_d y,\;\;  x,y\in A.
$$
A direct check shows that $(A,\lhd_1,\lhd_2,\rhd_1,\rhd_2)$ is a \pnva.
\end{ex}

\begin{thm}\mlabel{thm:4.9}
Let $(A,\lhd_1,\lhd_2,\rhd_1,\rhd_2)$ be a \pnva  and $(A,
\prec,\succ)$ be the associated noncommutative Novikov algebra.  Then
$r:=\sum_{i=1}^n(e_i^\ast\otimes e_i-e_i\otimes e_i^\ast),$
is an antisymmetric solution of the NYBE in
the noncommutative Novikov algebra $A\ltimes_{L^*_{\rhd_1},\,-(L^*_{\lhd_1}+L^*_{\rhd_1})}^{-(R^*_{\lhd_2}+R^*_{\rhd_2}),\,R^*_{\lhd_2}}A^*$ in Proposition~\mref{pro:ld-rep}, where $\{e_1, \ldots, e_n\}$ is a linear basis of
$A$ and $\{e_1^\ast, \ldots, e_n^\ast\}$ is its dual basis. Moreover, $r$ is nondegenerate and $A\ltimes_{L^*_{\rhd_1},\,-(L^*_{\lhd_1}+L^*_{\rhd_1})}^{-(R^*_{\lhd_2}+R^*_{\rhd_2}),\,R^*_{\lhd_2}}A^*$ with the induced bilinear form $\frakB$ given by
\begin{eqnarray}\mlabel{eq:frakb}
\frakB(x+a^*, y+b^*)=\langle x, b^*\rangle-\langle a^*, y\rangle, \quad   x, y\in A, a^*, b^*\in
A^\ast,
\end{eqnarray} 
is a quasi-Frobenius noncommutative Novikov algebra.
\end{thm}

\begin{proof}
By Proposition~\mref{pro:id}, the identity map $\id$ is an $\mathcal{O}$-operator of $(A,\prec,\su)$ associated to the representation
$(A,L_{\lhd_1},R_{\lhd_2},L_{\rhd_1}, R_{\rhd_2})$. Since $\id:A\to A$ can be identified  with  the tensor $r_\id:=\sum_{i}e^*_i\ot e_i\in A^*\ot A$, we have $r=r_\id-\sigma(r_\id)$. So  by Theorem~\mref{thm:5.8}, $r$ is  an antisymmetric solution of the NYBE in the noncommutative Novikov algebra $A\ltimes_{L^*_{\rhd_1},\,-(L^*_{\lhd_1}+L^*_{\rhd_1})}^{-(R^*_{\lhd_2}+R^*_{\rhd_2}),\,R^*_{\lhd_2}}A^*$.

On the other hand, by $r=\sum_{i=1}^n(e_i^\ast\otimes e_i-e_i\otimes e_i^\ast)$, $(A\oplus A^*)^*=A^*\oplus A$ and Eq.~\meqref{eq:tensor}, for $e_j\in A, e_j^*\in A^*$, we get a linear isomorphism $r^\sharp: (A\oplus A^*)^*\to A\ot A^*$ given by
$$r^\sharp(e_j)=\sum_i\langle e_i^*,e_j\rangle e_i-\langle e_i,e_j\rangle e_i^*=e_j\quad\text{and}\quad r^\sharp(e_j^*)=\sum_i\langle e_j^*,e_i^*\rangle e_i-\langle e_j^*, e_i\rangle e_i^*=-e_j^*.$$
Thus, $r$ is nondegenerate,  and this leads to $(r^\sharp)^{-1}(e_j)=e_j$ and $ (r^\sharp)^{-1}(e_j^*)=-e_j^*$.
So we obtain
$$\langle (r^\sharp)^{-1}(x+a^*),y+b^*\rangle=\langle -a^*+x,y+b^*\rangle=-\langle a^*,y\rangle+\langle x,b^*\rangle,$$
which gives
$\frakB(x+a^*, y+b^*)=\langle (r^\sharp)^{-1}(x+a^*),y+b^*\rangle$. Then the second statement follows by Proposition~\mref{pro:4.6}.
\end{proof}

By \cite[Theorem 2.10]{LH24}, a quasi-Frobenius Novikov algebra leads naturally to a pre-Novikov algebra structure. We now extend this fact to the noncommutative case.
\begin{thm}
Let $(A,\prec,\su,\frakB)$ be a quasi-Frobenius \nva. Then there is a  \pnva structure on $A$ given by
\begin{align*}
\frakB(x\lhd_1 y,z)&:= \frakB(z\ob x,y),\;
\frakB(x\lhd_2 y,z):= \frakB(x,y\su z),\\
\frakB(x\rhd_1 y,z)&:= \frakB(y,z\prec x),\;
\frakB(x\rhd_2 y,z):= \frakB(y\ob z,x),\quad x,y,z\in A,
\end{align*}
such that $(A,\prec,\su)$ is the associated \nva of $(A,\lhd_1, \lhd_2,\rhd_1, \rhd_2)$. 

\end{thm}
\begin{proof}
By the nondegenerate bilinear form $\frakB$ on $A$, we  get an invertible linear map $T:A^*\rightarrow A$ given by
\begin{eqnarray}
\frakB(x,y)=\langle T^{-1}(x), y\rangle, \quad x,y\in A. \mlabel{eq:T}
\end{eqnarray}
For all $x,y,z\in A$, we have
\begin{eqnarray*}
\frakB(x\lhd_1 y,z)&=&\frakB(z\ob x,y)=-\frakB(y,z\ob x )=-\langle T^{-1}(y),z\ob x\rangle\\
&=&-\langle T^{-1}(y),R_\ob( x)z\rangle=-\langle R^*_\ob( x)T^{-1}(y),z\rangle\\
&=&\frakB(T(-R^*_\ob( x)T^{-1}(y)),z).
\end{eqnarray*}
\begin{eqnarray*}
	\frakB(x\lhd_2 y,z)&=&\frakB(x,y\su z)= \langle T^{-1}(x),y\su z\rangle\\
	&=& \langle T^{-1}(x),L_\su(y)z\rangle= \langle L_\su^*(y)T^{-1}(x),z\rangle\\
	&=&\frakB(T(L_\su^*(y)T^{-1}(x)),z).
\end{eqnarray*}
\begin{eqnarray*}
	\frakB(x\rhd_1 y,z)&=&\frakB(y,z\prec x)= \langle T^{-1}(y),z\prec x\rangle\\
	&=& \langle T^{-1}(y),R_\prec(x)z\rangle= \langle R^*_\prec(x)T^{-1}(y),z\rangle\\
	&=&\frakB(T(R^*_\prec(x)T^{-1}(y)),z).
\end{eqnarray*}
\begin{eqnarray*}
	\frakB(x\rhd_2 y,z)&=&\frakB(y\ob z,x)=-\frakB(x,y\ob z )=-\langle T^{-1}(x),y\ob z\rangle\\
	&=&-\langle T^{-1}(x),L_\ob(y)z\rangle=-\langle L^*_\ob(y)T^{-1}(x),z\rangle\\
	&=&\frakB(T(-L^*_\ob(y)T^{-1}(x)),z).
\end{eqnarray*}
By the nondegeneracy  of $\frakB$, we obtain
\begin{eqnarray*}
x\lhd_1 y=T(-R^*_\ob(x)T^{-1}(y)),\quad x\lhd_2 y=T(L_\su^*( y)T^{-1}(x)),\\
x\rhd_1 y=T(R^*_\prec(x)T^{-1}(y)), \quad x\rhd_2 y=T(-L^*_\ob(y)T^{-1}(x)).
\end{eqnarray*}
For all $a^*,b^*\in A^*$, define
\begin{align*}
a^*\lhd^*_1 b^*&:=(-R^*_\ob)(T(a^*))b^*,\quad a^*\lhd^*_2 b^*:=L_\su^*(T(b^*))a^*,\\
a^*\rhd^*_1 b^*&:=R^*_\prec(T(a^*))b^*, \quad a^*\rhd^*_2 b^*: =T(-L^*_\ob)(T(b^*))a^*.
\end{align*}
Set $x=T(a^*)$, $y=T(b^*)$. Then
\begin{align*}
x\lhd_1 y&=T(a^*)\lhd_1 T(b^*)=T(a^*\lhd_1^* b^*),\;
 x\lhd_2 y=T(a^*)\lhd_2 T(b^*)=T(a^*\lhd_2^* b^*),\\
x\rhd_1 y&=T(a^*)\rhd_1 T(b^*)=T(a^*\rhd_1^* b^*), \;
x\rhd_2 y=T(a^*)\rhd_2 T(b^*)=T(a^*\rhd_2^* b^*).
\end{align*}
To show that $(A,\lhd_1, \lhd_2,\rhd_1, \rhd_2)$ is a \pnva, it suffices to verify that $(A^*,\lhd_1^*, \lhd_2^*,\rhd_1^*, \rhd_2^*)$ is a \pnva.
For all  $a^*,b^*,c^*\in A^*$, we have
\begin{eqnarray*}
&&\langle c^*,T(a^*)\prec T(b^*)-T((-R^*_\ob)(T(a^*))b^*+L_\su^*(T(b^*))a^*) \rangle\\
&=&\frakB(T(c^*),T(a^*)\prec T(b^*))+\frakB(T(c^*),T(R^*_\ob(T(a^*))b^*-L_\su^*(T(b^*))a^*))\\
&=&\frakB(T(c^*),T(a^*)\prec T(b^*))-\frakB(T(R^*_\ob(T(a^*))b^*-L_\su^*(T(b^*))a^*),T(c^*))\\
&=&\frakB(T(c^*),T(a^*)\prec T(b^*))-\langle R^*_\ob(T(a^*))b^*-L_\su^*(T(b^*))a^*, T(c^*)\rangle\\
&=&\frakB(T(c^*),T(a^*)\prec T(b^*))-\langle b^*,T(c^*)\ob T(a^*)\rangle+\langle a^*,T(b^*)\su T(c^*)\rangle\\
&=&\frakB(T(c^*),T(a^*)\prec T(b^*))-\frakB (T(b^*),T(c^*)\ob T(a^*)) +\frakB (T(a^*),T(b^*)\su T(c^*))\\
&=&0
\end{eqnarray*}
by Eq.~\meqref{eq:109}. Similarly, we obtain
\begin{eqnarray*}
	&&\langle c^*,T(a^*)\su T(b^*)-T(R^*_\prec(T(a^*))b^*+(-L^*_\ob)(T(b^*))a^*) \rangle\\
	&=&\frakB(T(c^*),T(a^*)\su T(b^*))-\frakB(T(c^*),T(R^*_\prec(T(a^*))b^*-L_\ob^*(T(b^*))a^*))\\
	&=&\frakB(T(c^*),T(a^*)\su T(b^*))+\frakB(T(R^*_\prec(T(a^*))b^*-L_\ob^*(T(b^*))a^*),T(c^*))\\
	&=&\frakB(T(c^*),T(a^*)\su T(b^*))+\langle R^*_\prec(T(a^*))b^*-L_\ob^*(T(b^*))a^*, T(c^*)\rangle\\
	&=&\frakB(T(c^*),T(a^*)\su T(b^*))+\langle b^*,T(c^*)\prec T(a^*)\rangle-\langle a^*,T(b^*)\ob T(c^*)\rangle\\
	&=&\frakB(T(c^*),T(a^*)\su T(b^*))+\frakB (T(b^*),T(c^*)\prec T(a^*)) -\frakB (T(a^*),T(b^*)\ob T(c^*))\\
	&=&0
\end{eqnarray*}
by Eq.~\meqref{eq:109}. This gives
\begin{align*}
	T(a^*)\prec T(b^*)-T((-R^*_\ob)(T(a^*))b^*+L_\su^*(T(b^*))a^*)&=0,\\
	T(a^*)\su T(b^*)-T(R^*_\prec(T(a^*))b^*+(-L^*_\ob)(T(b^*))a^*)&=0.
\end{align*}
So $T:A^*\rightarrow A$ is an invertible $\calo$-operator of $(A,\prec,\su)$ associated to $(A,-R_\ob^*,L^*_\su,R^*_\prec,-L^*_\ob)$. By Proposition~\mref{pro:pnva}, $(A^*,\lhd_1^*, \lhd_2^*,\rhd_1^*, \rhd_2^*)$ is a \pnva. Moreover,
\begin{eqnarray*}
a\lhd_1 b+a \lhd_2 b = T((-R^*_\ob)(T(a^*))b^*+L_\su^*(T(b^*))a^*)=T(a^*)\prec T(b^*) =a\prec b.\\
a\rhd_1 b+a \rhd_2 b = T(R^*_\prec(T(a^*))b^*+(-L^*_\ob)(T(b^*))a^*)=T(a^*)\su T(b^*) =a\su b.
\end{eqnarray*}
Thus, $(A,\prec,\su)$ is the associated \nva of $(A,\lhd_1, \lhd_2,\rhd_1, \rhd_2)$.
\end{proof}

\section{Differential ASI bialgebras \wt, their characterizations and derived structures}
\mlabel{sec:DASI}
This section first gives  the notion of a differential \asi bialgebras of weight $\lambda$, based on the concept of an admissible differential algebra \wt. Several examples of differential \asi bialgebras of weight $\lambda$ are given. Using semi-direct product differential algebras, the notion of a representation of a differential algebra \wt is obtained. As a result, an admissible differential algebra of weight $\lambda$ is shown to be equivalent to a special representation of a differential algebra of weight $\lambda$.
\subsection{Differential ASI bialgebras \wt}
\begin{defi} Let $\lambda\in \bfk$ be given.
\begin{enumerate}
\item Let $(A,\cdot)$ be an algebra. A linear map $\partial:A \rightarrow A$ is called a {\bf derivation \wt} or simply a {\bf derivation}  if
\begin{equation}\mlabel{eq:rb}
\partial(x\cdot y)=\partial(x)\cdot y+x\cdot \partial(y)+\lambda \partial(x)\cdot \partial(y),\quad x, y\in A.
\end{equation}
In this case,    the triple $(A,\cdot,\pa)$ is called a {\bf \da \wt}, or simply a {\bf \da}.
\item
	Let $(A,\Delta)$ be a coalgebra. A linear map $\dh:A \rightarrow A$ is called a {\bf coderivation \wt} or simply a {\bf coderivation} if
\begin{equation}\mlabel{eq:corb}
	\Delta \dh=(\dh\otimes \id+\id \otimes \dh +\lambda \dh\otimes \dh)\Delta.
\end{equation}
In this case,    the triple $(A,\Delta,\dh)$ is called a {\bf \dca \wt}, or simply a {\bf \dca}.
\end{enumerate}
\end{defi}

\begin{defi} \mlabel{de:corb}
Let $\lambda\in \bfk$ be given.
\begin{enumerate}
\item
  An {\bf \adm \da \wt} or simply an  {\bf \adm \da} is  a quadruple $(A,\cdot,\pa,\dh)$, consisting of a \da $(A,\cdot,\pa)$  \wt and  a linear map $\dh:A\to A$ such that
\begin{align}
	\dh(x)\cdot y&=x\cdot\partial(y)+\dh(x\cdot y)+\lambda\dh(x\cdot\partial(y)),
	\mlabel{eq:pduqr} \\
	x\cdot\dh (y)&=\partial(x)\cdot y+\dh(x\cdot y)+\lambda\dh(\partial(x)\cdot y), \quad x, y\in A.
	\mlabel{eq:pduql}
\end{align}
\item
An {\bf \adm \dca \wt} or simply an {\bf \adm \dca} is a quadruple $(A,\Delta,\dh,\partial)$, consisting of a \dca $(A,\Delta,\dh)$ \wt  and a linear map $\partial:A \rightarrow A $ such that
\begin{align}
 (\partial \otimes \id)\Delta&=(\id\otimes\dh)\Delta+\Delta\partial+\lambda(\id\otimes\dh)\Delta\partial,
\mlabel{eq:pduqrd}\\
(\id\otimes\partial)\Delta&=(\dh\otimes \id)\Delta+\Delta\partial+\lambda(\dh\otimes \id)\Delta\partial.
\mlabel{eq:pduqld}
\end{align}
\end{enumerate}
\end{defi}
For a finite-dimensional vector space
$A$,  let $\dh:A \to A $ and $\Delta:A\to  A\ot A$ be linear maps. Denote by $\circ: A^*\ot A^*\to A^*$ the linear dual of $\Delta$. Then
$(A,\Delta,\dh)$ is a differential coalgebra \wt if and only if
$(A^*, \circ,\dh^*)$ is a differential algebra \wt.
Furthermore, we have
\begin{lem}\mlabel{lem:2.3}
With the notations as above,  $(A,\Delta,\dh,\partial)$ is an \adm \dca \wt if and only
if  $(A^*,\circ,\dh^*,\partial^*)$ is an \adm \da \wt.
\end{lem}
\begin{proof}
It follows from a straightforward computation.
\end{proof}

Recall from~\cite{Bai1} the notion of  an \asi bialgebra.
\begin{defi}
\mlabel{de:bial}
An {\bf antisymmetric infinitesimal bialgebra} or
simply an {\bf \asi bialgebra} is a triple $(A,\cdot,\Delta)$
consisting of a vector space $A$ and linear maps $\cdot: A\ot A\to
A$ and $\Delta:A\to A\ot A$ such that
\begin{enumerate}
\item
the pair $(A,\cdot)$ is an associative algebra,
\item
the pair $(A,\Delta)$ is a coassociative coalgebra, and
\item
the following compatibility conditions hold.
 \begin{align}
 \Delta(x\cdot y)&=(R_A(y)\otimes \id)\Delta(x)+(\id\otimes L_A(x))\Delta(y),
 \mlabel{eq:3.14}\\
 (L_A(x)\otimes \id-\id\otimes R_A(x))\Delta(y)&=\sigma(\id\otimes R_A(y)-L_A(y)\otimes \id)\Delta(x), \quad  x, y\in A.
 \mlabel{eq:3.15}
 \end{align}
\end{enumerate}
\end{defi}

We now extend  the concept of an \asi bialgebra to the context of differential algebras \wt.

\begin{defi} \mlabel{de:dasib}
A {\bf differential antisymmetric infinitesimal bialgebra \wt}, or simply a {\bf differential \asi bialgebra} is a quintuple $(A,\cdot,\Delta,\partial,\dh)$ consisting of a vector space $A$ and linear maps
$$\cdot: A\ot A\to A,\ \ \Delta:A\to A\ot A,\ \  \partial,\dh:A\to A$$
such that
\begin{enumerate}
\item
$(A,\cdot,\Delta)$ is an \asi bialgebra,
\mlabel{it:rbasi1}
\item
$(A,\cdot,\partial,\dh)$  is an \adm \da \wt, and
\mlabel{it:rbasi2}
\item
$(A,\Delta,\dh,\partial)$ is an \adm \dca \wt.
\mlabel{it:rbasi3}
\end{enumerate}
\end{defi}

Here are two preliminary examples.
\begin{ex}
For a given scalar $0\neq\lambda\in\bfk$  and any \ASIb $(A,\cdot,\Delta)$,  the quintuple $(A,\cdot,\Delta,-\lambda^{-1}\id,-\lambda^{-1}\id)$ is a \dASIb \wt.
\end{ex}

\begin{ex}
	Let $(A,\cdot_A,\Delta_A)$ and $(B,\cdot_B,\Delta_B)$ be ASI bialgebras. Define
	$$\cdot:=\cdot_A+\cdot_B:(A\oplus B)\ot (A\oplus B)\to A\oplus B, \quad(a_1+b_1)\ot (a_2+b_2)\mapsto a_1\cdot_A a_2+b_1\cdot_B b_2$$
	and
	\begin{eqnarray*}
		&&\Delta:=\Delta_A+\Delta_B:A\oplus B\to (A\oplus B)\ot (A\oplus B),\\
		&&\quad\quad\quad\quad(a+b)\mapsto \Delta_A(a)+\Delta_B(b)= \sum_{a}a_{(1)}\ot a_{(2)}+\sum_{b}b_{(1)}\ot b_{(2)}\in (A\oplus B)\ot (A\oplus B).
	\end{eqnarray*}
	Then by ~\cite[Example ~2.6]{BGM24}, $(A \oplus B,\cdot,\Delta)$ is an ASI bialgebra. Let $\partial_A , \partial_B : A \oplus B \rightarrow A \oplus B$ be the projections to $A$ and $B$, respectively.
	First,  we compute $$\partial_A((a_1+b_1)\cdot (a_2+b_2))=a_1\cdot_A a_2=\partial_A(a_1+b_1)\cdot (a_2+b_2)+(a_1+b_1)\cdot\partial_A(a_2+b_2) -\partial_A(a_1+b_1)\cdot\partial_A(a_2+b_2),$$
	proving that $(A\oplus B, \cdot, \partial_A)$ is a \da of weight $-1$.
	
	Second,  we verify
	$$\Delta\partial_A(a+b)=\Delta_A(a)=(\partial_A \ot \id_{A\oplus B}+\id_{A\oplus B}\ot\partial_A-\partial_A \ot\partial_A)\Delta (a+b),$$
	showing that $(A\oplus B, \Delta, \partial_A)$ is a \dca of weight $-1$.
	
Third, the identity $$\partial_A(a_1+b_1)\cdot (a_2+b_2)=a_1\cdot_Aa_2=(a_1+b_1)\cdot\partial_A (a_2+b_2)+\partial_A((a_1+b_1)\cdot (a_2+b_2))-\partial_A\Big((a_1+b_1)\cdot\partial_A (a_2+b_2)\Big),$$
yields the validation of Eq.~\meqref{eq:pduqr}. 
	
Similarly, we can get Eq.~\meqref{eq:pduql}. Since
	$$( \partial_A\ot\id_{A\oplus B})\Delta (a+b)=\Delta_A(a)=((\id_{A\oplus B}\ot\partial_A)\Delta+\Delta\partial_A- (\id_{A\oplus B}\ot\partial_A)\Delta\partial_A)(a+b),$$
	we have Eq.~\meqref{eq:pduqrd}. Similarly, Eq.~\meqref{eq:pduqld} holds.
	Thus the quintuple $(A \oplus B,\cdot_A + \cdot_B,\Delta_A + \Delta_B,\partial_A,\partial_A)$ is a \dASIb of weight $-1$. By the same way, we can prove that $(A \oplus B,\cdot_A + \cdot_B,\Delta_A + \Delta_B,\partial_B,\partial_B)$ is a \dASIb of weight $-1$.
\end{ex}

\subsection{Representations of \das \wt}

Let $A$ be an algebra. A {\bf representation} of $A$ is a triple $(V,\ell,r)$, abbreviated as $V$, consisting of a vector space $V$ and linear maps
$\ell, r: A\to  \End (V)$
such that
\begin{equation}\mlabel{eq:11rep}
\ell(xy)v=\ell(x)\ell(y)v,\,
r(xy)v=r(y)r(x)v,\,
 \ell(x)r(y)v=r(y)\ell(x)v, \quad x,y\in A, v\in V.
\end{equation}

For instance, the triple $(A, L, R)$ is a representation of $A$, called the {\bf adjoint representation} of $A$.
Let $\ell, r: A\to \End (V)$ be linear maps.
Define a multiplication on $A\oplus V$ by
 \begin{equation}
 (x+u)\at (y+v):=xy+(\ell(x)v+r(y)u ),\quad x, y\in A, u, v \in V.
 \mlabel{eq:2.3}
 \end{equation}
As is well known, $A\oplus V$ is an
algebra, denoted by $A\ltimes_{\ell,r} V$ and called
the {\bf semi-direct product algebra} of $A$ by $V$, if and only if $(V,\ell,r)$ is a
representation of $A$.

\begin{defi}\mlabel{de:2.8}Let $(A,\cdot,\alpha)$ be a \da \wt.
\begin{enumerate}
\item
A {\bf representation} of a \da  $(A,\cdot,\partial)$  is a quadruple $(V, \ell, r,
\theta)$, where $(V, \ell, r)$ is a representation of $A$ and  $\theta:V \rightarrow V$ is a
 linear map such that
 \begin{align}
 \theta(\ell(x)v)&=\ell(\partial(x))v+\ell(x)\theta(v)+\lambda \ell(\partial(x))\theta(v),
\mlabel{eq:2.1}\\
 \theta(r(x)v)&=r(\partial(x))v+r(x)\theta(v)+\lambda r(\partial(x))\theta(v), \quad x\in A, v\in V. \mlabel{eq:2.2}
 \end{align}
\item\mlabel{it:282}
Two representations $(V_1, \ell_1, r_1, \theta_1)$ and $(V_2,
\ell_2, r_2, \theta_2)$ of a \da $(A,\cdot, \partial)$  are
called {\bf equivalent} if there exists a linear
isomorphism $\varphi:V_1\rightarrow V_2$ such that
 \begin{equation} \mlabel{eq:282}
 \varphi (\ell_1(x)v)=\ell_2(x)\varphi(v),\, \varphi(r_1(x)v)=r_2(x)\varphi(v),\, \varphi (\theta_1 v)=\theta_2\varphi(v) ,\quad x\in A, v\in V_1.
 \end{equation}
\end{enumerate}
 \end{defi}

\begin{pro}\mlabel{pro:2.2}
Let $(A,\partial)$ be a \da \wt.
Let
$(V, \ell, r)$ be a representation of the associative algebra $A$ and let
 $\theta:V \rightarrow V$ be a linear map.
Define a linear map
 \begin{equation}
 \partial_{A\oplus V}: A\oplus V\to  A\oplus V, \quad  \partial_{A\oplus V}(x+u):=\partial(x)+\theta(u),
\mlabel{eq:2.4}
 \end{equation}
which is simply denoted by $\partial_{A\oplus V}:=\partial+\theta$.
 Then together with the multiplication defined in Eq.~\meqref{eq:2.3}  , $(A\oplus V, \partial_{A\oplus V})$ is a \da  if and only if $(V,\ell,r,\theta)$ is a representation of $(A,\partial)$.
 The resulting \da  is denoted by $(A\ltimes_{\ell,r} V, \partial+\theta)$ and  called the {\bf semi-direct product differential algebra} of $(A,\partial)$ associated to $(V,\ell,r,\theta)$.
 \end{pro}
\begin{proof}	It follows from a straightforward verification.
\end{proof}

\begin{lem}
\mlabel{lem:admrep}
Let  $(A,\partial)$ be a \da \wt.  Let $(V, \ell, r)$ be a representation of the algebra $A$ and let $\theta:V\to V$ be a linear map.
The quadruple $(V^*,r^*,\ell^*,\theta^*)$ is a representation of $(A,\partial)$ if and only if  $\theta$ satisfies
\begin{eqnarray}
& r(x)\theta(v)-r(\partial(x))v-\theta(r(x)v)-\lambda\theta(r(\partial(x))v)=0,& \mlabel{eq:it:2.3a} \\
& \ell(x)\theta(v)-\ell(\partial(x))v-\theta(\ell(x)v)-\lambda \theta(\ell(\partial(x))v)=0,&  x\in A, v\in V.
\mlabel{eq:it:2.3b}
\end{eqnarray}
\end{lem}

\begin{proof}
Since $(V^*, r^*, \ell^*)$ is a representation of $A$, we just need to show that $\theta^*$ satisfies Eqs.~\meqref{eq:2.1} and ~\meqref{eq:2.2}.
Indeed, we have
\vspace{-.2cm}
\begin{eqnarray*}
&&\langle \theta^*(r^*(x)v^*)-r^*(\partial(x))v^*-r^*(x)\theta^* (v^*)-\lambda r^*(\partial(x))\theta^*(v^*), v \rangle \\
&=& \langle  v^*, r(x)\theta(v)-r(\partial(x))v-\theta(r(x)v)-\lambda\theta(r(\partial(x))v)\rangle,
\quad x\in A, v\in V, v^*\in V^*.
\end{eqnarray*}
Thus, Eq.~\meqref{eq:2.1} is equivalent to
Eq.~\meqref{eq:it:2.3a}. The same argument proves that Eq.~\meqref{eq:2.2} is equivalent to
Eq.~\meqref{eq:it:2.3b}.
\end{proof}

\begin{pro}\mlabel{pro:2.11}
\begin{enumerate}
\item\mlabel{it:211}
Let $(A,\cdot,\pa)$ be a \da \wt. Then $(A,L_\cdot,R_\cdot,\pa)$ is a representation of $(A,\cdot,\pa)$,  called the {\bf adjoint representation} of $(A,\cdot,\partial)$.
\item \mlabel{it:212}
Furthermore, let $\dh:A\to A$ be a linear map. Then $(A,\cdot,\pa,\dh)$ is an  \adm \da \wt if and only if $(A^*, R^*_\cdot, L^*_\cdot, \dh^*)$ is a representation of  $(A,\cdot,\pa)$.
\end{enumerate}
\end{pro}
\begin{proof}(\mref{it:211}) follows by the fact that $\pa$ is a derivation \wt.

\smallskip
\noindent
(\mref{it:212}) follows from Lemma~\mref{lem:admrep} and Eqs.~\meqref{eq:pduqr} and \meqref{eq:pduql}.
\end{proof}

\subsection{Characterizations of  \dASIbs \wt}
\mlabel{subsec:char}
This section presents the notion of a matched pair of differential algebras \wt, which is equivalent to a direct sum of two differential algebras \wt. Then the concept of a  double construction of differential Frobenius algebras \wt is given. As a consequence, any differential \asi bialgebra \wt is proved to be equivalent to some special  matched pair of differential algebras \wt and a double construction of  differential Frobenius algebra \wt, respectively. Under two compatibility conditions, we show that a differential \asi bialgebra of weight $0$ induces a \nvbia.

We first recall the concept of a matched pair of algebras.

\begin{defi}\mlabel{de:match}
A {\bf matched pair of algebras} consists of
algebras $(A, \opa)$ and $(B, \opb)$, together with
linear maps $\ell_A, r_A: A\to  \End (B)$ and $\ell_B,
r_B: B\to  \End (A)$ such that
\begin{enumerate}
\item
$(A, \ell_B, r_B)$ is a representation of $(B, \opb)$, 
\item
$(B, \ell_A, r_A)$ is a representation of $(A, \opa)$, and
\item
 the following compatibility conditions hold: for all $a, a'\in A$ and $b, b'\in B$,
 \begin{eqnarray*}
& \ell_A(a)(b\opb b')=\ell_A(r_B(b)~a )b'+(\ell_A(a)b)\opb b',&
\\
& r_A(a)~(b\opb b')=r_A(\ell_B(b')a)~b+b\opb (r_A(a)~b' ),&
\\
& \ell_B(b)(a\opa a')=\ell_B(r_A(a)~b)a'+(\ell_B(b)a)\opa a',&
\\
& r_B(b)(a\opa a')=r_B(\ell_A(a')b)~a+a\opa (r_B(b)~a'),&
 \\
& \ell_A(\ell_B(b)a)b'+(r_A(a)~b)\opb b'=r_A(r_B(b')~a)~b+b\opb (\ell_A(a)b'),&
\\
& \ell_B(\ell_A(a)b)a'+(r_B(b)~a)\opa a'=r_B(r_A(a')~b)~a+a\opa (\ell_B(b)a').&
\end{eqnarray*}
\end{enumerate}
\end{defi}
With the notations as above, we define a multiplication on the direct
sum $A\oplus B$ by
 \begin{equation}
 (a+b)\star (a'+b'):=(a\opa a'+r_B(b')a +\ell_B(b) a')+(b\opb b'+\ell_A(a) b'+r_A(a')b ).
\mlabel{eq:3.7}
 \end{equation}
Then $(A\oplus B, \star)$ is an algebra if and only if $((A,\opa),(B,\opb), \ell_A, r_A,$ $\ell_B,r_B)$ is a matched pair of $(A,\opa)$ and $(B,\opb)$.
The resulting algebra $(A\oplus B,\star)$ is denoted by $A\bowtie_{\ell_A,r_A}^{\ell_B,r_B} B$ or simply
$A\bowtie B$.
We now give the notion of a matched pair of \das \wt.
\begin{defi}\mlabel{de:3.3}
A {\bf matched pair of \das \wt} is a sextuple  $((A, \partial_A), (B, \partial_B),$ $\ell_A, r_A,$ $\ell_B, r_B)$, where $(A, \partial_A)$, $(B, \partial_B)$ are \das \wt, $(B, \ell_A, r_A, \partial_B)$ is a representation of $(A,\partial_A)$, $(A, \ell_B, r_B$, $\partial_A)$ is a representation of $(B, \partial_B)$,
and $(A,B, \ell_A,r_A,\ell_B,r_B)$ is a matched pair of algebras.
\end{defi}

There is a characterization of \asi bialgebras by matched pairs of algebras.
\vspace{-.1cm}
\begin{thm}\cite[Theorem 2.2.3]{Bai1} \mlabel{thm:md}
Let $(A, \cdot)$ be an algebra. Suppose that there is an algebra
$(A^*, \circ)$ on the linear dual $A^*$. Let $\Delta:A\to
A\ot A$ be the linear dual of $\circ$. Then $(A,\cdot, \Delta)$ is an \asi bialgebra if and only if
$(A,A^*, {R_\cdot}^*, {L_\cdot}^*, {R_\circ}^*,
{L_\circ}^*)$ is a matched pair of algebras.
\end{thm}

We extend this property to the context of \das \wt.
\begin{thm}
Let $(A,\partial_A)$ and $(B,\partial_B)$ be \das \wt and let $(A,B, \ell_A,$ $ r_A,\ell_B, r_B)$ be a matched pair of the algebras $A$ and $B$. On the resulting algebra $A\bowtie B$ from Eq.~\meqref{eq:3.7}, define the linear map
 \begin{equation}
\partial_{A\oplus B}:A\bowtie B\to A\bowtie B,  \quad \partial_{A\oplus B}(a+b):=\partial_A(a)+\partial_B(b),\quad a\in A, b\in B.
 \mlabel{eq:3.8}
 \end{equation}
Then the pair $(A\bowtie B,\partial_{A\oplus B})$ is a
\da \wt if and only if $((A,
\partial_A),$ $ (B, \partial_B), \ell_A, r_A,$ $\ell_B, r_B)$ is a matched pair
of  $(A,\partial_A)$ and $(B,\partial_B)$.
\mlabel{thm:3.4}
 \end{thm}
\vspace{-.2cm}
\begin{proof}
It is similar to the proof of ~\cite[Theorem 2.13]{LLB23}.
\end{proof}

We now give the relation between differential \asi bialgebras between matched pairs of differential algebras.
\begin{thm} \mlabel{thm:3.5}
Let $(A, \cdot,\partial)$ be a \da \wt. Suppose that
 $(A^*, \circ,\dh^*)$ is a \da \wt. Let $\Delta:A\rightarrow A\otimes A$ be the linear dual of
$\circ$.
Then the quintuple
$(A,\cdot,\Delta,\partial,\dh)$ is a differential \asi bialgebra \wt if and only if
the sextuple $((A,\partial),(A^*,\dh^*), {R_\cdot}^*, {L_\cdot}^*, {R_\circ}^*, {L_\circ}^*)$ is a matched pair of \das \wt.
\end{thm}
\begin{proof}
It follows from \cite[Theorem 3.13]{LLB23}.
\end{proof}

We recall the concept of a double construction of Frobenius algebras. See~\mcite{Bai1} for details.
\begin{defi}
 A bilinear form $\frakB$ on an algebra $A$ is called {\bf invariant} on $A$ if
\begin{equation}
 \mathfrak{B}(xy, z)=\mathfrak{B}(x,yz),\quad x, y, z\in A.
 \mlabel{eq:1.3}
 \end{equation}
A {\bf Frobenius algebra} is a pair $(A,\frakB)$ consisting of an algebra $A$ and a
nondegenerate invariant bilinear form $\frakB$. A Frobenius
algebra $(A,\frakB)$ is called {\bf symmetric} if $\frakB$ is symmetric.
 \mlabel{de:1.4}
 \end{defi}

Let $(A,\cdot)$ be an algebra. Suppose that there is an algebra
structure $\circ$ on its dual space $A^\ast$, and an algebra
structure $\star$ on the direct sum $A\oplus A^\ast$ of the underlying
vector spaces of $A$ and $A^\ast$ which contains both $(A,\cdot)$
and $(A^\ast,\circ)$ as subalgebras. Define a
bilinear form $\frakB_d$ on $A\oplus A^*$ by
\vspace{-.2cm}
\begin{equation}
\frakB_d(x + a^* ,y + b^* ) = \langle x,b^*\rangle + \langle a^* , y\rangle, \quad  a^*, b^* \in A^* , x, y \in A.
\mlabel{eq:3.9}
\end{equation}
If $\frakB_d$ is invariant on $(A\oplus A^*,\star)$, so that
$(A\oplus A^*,\star,\frakB_d)$ is a symmetric Frobenius algebra, then the Frobenius algebra is called a {\bf double construction of Frobenius algebras} associated to $(A,\cdot)$ and
$(A^\ast,\circ)$, denoted by $(A\bowtie A^*, \star,\frakB_d)$.
\begin{thm}
\cite[Theorem~2.2.1]{Bai1} Let $(A,\cdot)$ and $(A^*,\circ)$ be algebras.
Then there is a double construction of
Frobenius algebras associated to $(A,\cdot)$ and $(A^*,\circ)$ if
and only if $(A,A^*, R^*_\cdot, L^*_\cdot,R^*_\circ,L^*_\circ)$ is
a matched pair of algebras.
\mlabel{thm:frob}
\end{thm}

\begin{defi}\mlabel{de:1.3}
A {\bf differential Frobenius algebra \wt} is a triple $(A,\partial,\frakB)$
where $(A,\partial)$ is a \da \wt and $(A,\frakB)$ is a Frobenius algebra. Define a linear map $\hat\partial :A\to A$ by 
\begin{equation}\mlabel{eq:adjoint}
\mathfrak{B}(\partial (x), y)=\mathfrak{B}(x, \hat \partial (y)),\quad x,y\in A,
\end{equation}
which is called the {\bf adjoint} of $\pa$.
\end{defi}

\begin{pro}\mlabel{pp:3.9}
\begin{enumerate}
\item\mlabel{it:391}
Let $(A,\cdot,\partial,\frakB)$ be a differential symmetric Frobenius algebra \wt. Let $\hat\partial$ be the adjoint of $\partial$. Then
the quadruple $(A^*, R^*,L^*,{\hat \partial}^*)$ is a representation of the \da $(A,\cdot,\partial)$, which is equivalent to $(A,L_\cdot,R_\cdot, \pa)$. Furthermore, $(A,\cdot,\pa,\hat{\pa})$ is an \adm \da\wt.
\item\mlabel{it:392}
Let $(A,\cdot,\partial,\dh)$ be an \adm \da \wt.  If the resulting representation $(A^*, R^*,L^*,\dh^*)$ of $(A,\cdot,\partial)$ is equivalent to $(A,L,R,\partial)$, then there exists a nondegenerate bilinear form
$\frakB$ such that $(A,\partial,\frakB)$ is a differential Frobenius algebra \wt for which $\hat \partial=\dh$.
\end{enumerate}
\end{pro}

\begin{proof}(\mref{it:391})
Suppose that $(A,\cdot,\partial,\frakB)$ is a differential symmetric Frobenius algebra \wt. Then for all $x,y,z\in A$, we have
\begin{eqnarray*}
0&=& \frakB(\partial (xy),z)-\frakB(\partial (x)y,z)-\frakB(x\partial (y)),z)-\frakB(\lambda \partial (x)\partial (y),z)\\
&=& \frakB(xy,\hat\partial (z))- \frakB(\partial (x),yz) - \frakB(x,\partial (y)z) - \frakB(\lambda \partial (x),\partial (y)z)\\
&=& \frakB(x,y\hat\partial (z)- \hat \partial (xz) - \partial (y)z - \lambda \hat\partial (\partial (y)z),
\end{eqnarray*}
yielding
$ y\hat\partial (z)= \hat \partial (yz) + \partial (y)z + \lambda \hat\partial (\partial (y)z=0.$
This gives Eq.~(\mref{eq:pduql}). Applying the symmetry of
$\frakB$, a similar argument gives Eq.~(\mref{eq:pduqr}).
Then by Proposition~\mref{pro:2.11}(\mref{it:212}),
$(A^*, R^*,L^*,{\hat \partial}^*)$ is a representation of $(A,\partial)$.

Define a linear map $\phi:A\to  A^*$ by
$$\phi(x)(y):=\langle \phi(x), y\rangle=\frakB(x,y),\quad x,y\in A.$$
The nondegeneracy of $\frakB$ gives the bijectivity of $\phi$.
For all $x,y,z\in A$, we have
\begin{eqnarray*}
\phi(\pa(x))y=\langle\phi( \pa(x)),y\rangle=\frakB(\pa(x),y)=\frakB(x,\hat{\pa} (y))=\langle\phi(x),\hat{\pa}(y)\rangle=\hat{\pa}^*(\phi(x))y.
\end{eqnarray*}
\begin{eqnarray*}
\phi (L(x)y)z=\frakB(xy,z)=\frakB(zx,y)=\langle \phi(y), zx\rangle=\langle R^*(x)\phi(y), z\rangle=R^*(x)\phi(y) z.
\end{eqnarray*}
Similarly, $\phi(R(y)x)z=L^*(y)\phi(x)z$.
So by Eq.~\meqref{eq:282}, $(A, L, R, \partial)$ is equivalent
to $(A^*, R^*,L^*,{\hat \partial}^*)$ as representations of $(A,\cdot,\partial)$.
\smallskip

\noindent
(\mref{it:392})
Assume that $\phi:A\rightarrow A^*$ is an
isomorphism such that  $(A, L, R, \partial)$ and
$(A^*, R^*,L^*,\dh^*)$ are equivalent. Define a bilinear form $\frakB$ on
$A$ by
$$\frakB(x,y):=\langle \phi(x), y\rangle,\quad x,y\in A.$$
Then by the above proof, $\frakB$ is invariant on $(A,\cdot)$, and so  $(A,\cdot,\partial,\frakB)$ is a differential Frobenius algebra \wt. By $\phi(\pa(x))y=\dh^*(\phi(x))y$, we get $\langle \phi(\pa(x)),y\rangle=\langle \dh^*(\phi(x)),y\rangle$. Thus, $\frakB(x,\hat{\pa}(y))=\frakB(\pa(x),y)=\frakB(x,\dh(y)),$
and so $\hat \partial=\dh$.
\end{proof}

\begin{defi}   \mlabel{de:3.10}
 Let $(A, \cdot, \partial)$ be a \da \wt. Suppose that $(A^*,\circ, \dh^*)$ is a \da \wt.
 A {\bf double construction of differential Frobenius algebras} \wt associated to $(A, \cdot, \partial)$ and $(A^*,\circ, \dh^*)$, denoted by $(A\bowtie A^*,\star,\partial+ \dh^*, \frakB_d)$, is a double construction $(A\bowtie A^*, \star, \frakB_d)$ of Frobenius algebra associated to $(A, \cdot)$ and $(A^*,\circ)$ such that $(A\bowtie A^*,\star,\partial+ \dh^*, \frakB_d)$ is a  differential Frobenius algebra \wt.
 \end{defi}
 \begin{lem}
\mlabel{lem:3.11}
Let $(A\bowtie A^*,\star,\partial+\dh^*, \mathfrak{B}_d)$ be a double construction of differential Frobenius algebra \wt associated to $(A,\partial)$ and $(A^*,\dh^*)$. The adjoint $\widehat{ \partial+\dh^*}$ of $\partial+\dh^*$ with respect to $\frakB_d$ is $\dh+ \partial^*$.
\end{lem}

\begin{proof} For all $x,y\in A, a^*,b^*\in
A^*$, by Eq.~(\mref{eq:3.9}), we have 
\begin{align*}
\frakB_d\big((\partial +\dh ^*)(x+a^*),y+b^*\big)&=\frakB\big(\partial (x)+\dh ^*(a^*),y+b^*\big)
=\langle \partial (x),b^*\rangle + \langle \dh ^*(a^*), y\rangle\\
&=\langle x, \partial ^*(b^*)\rangle + \langle a^*, \dh (y)\rangle
=\frakB_d(x+a^*, (\dh +\partial ^*)(y+b^*)).
\end{align*}
So we get $\widehat{ \partial+\dh^*}=\dh+ \partial^*$.
\end{proof}

The subsequent theorem demonstrates that a double structure of differential Frobenius algebras \wt is fully characterized by  some matched pair of differential algebras \wt.
\begin{thm} Let $(A,\cdot,\partial)$ be a \da \wt. Suppose that $(A^*,\circ,\dh^*)$ is a \da \wt. Then $(A\bowtie A^*,\star,\partial+\dh^*,\frakB_d)$ is a double construction of differential Frobenius
algebras \wt  if and only if
$((A,\partial),(A^*,\dh^*), R^*_\cdot, L^*_\cdot,R^*_\circ,L^*_\circ)$ is a
matched pair of \das \wt. \mlabel{thm:3.12}
\end{thm}
\begin{proof}The proof follows that  of~\cite[Theorem 3.6]{LLB23}.
\end{proof}

Combining Theorems~\mref{thm:3.5} and~\mref{thm:3.12}, we have
\begin{cor}
Let $(A, \cdot,\partial)$ be a \da \wt. Suppose that  $(A^*, \circ,\dh^*)$ is
a \da \wt   on $A^*$. Let $\Delta:A\rightarrow A\otimes A$ be the linear dual of
the multiplication $\circ$. Then the following statements are equivalent.
 \begin{enumerate}
 \item The sextuple $((A,\partial),(A^*,\dh^*), {R_\cdot}^*, {L_\cdot}^*, {R_\circ}^*, {L_\circ}^*)$ is a matched pair of \das \wt.
\mlabel{it:rbb1}
\item The quadruple $\big(A\bowtie A^*, \star, \pa+ \dh^*, \frakB_d\big)$ is a double construction of differential Frobenius algebras \wt associated to $(A,\cdot,\partial)$ and
$(A^*,\circ,\dh^*)$. \mlabel{it:rbb2}
\item The quintuple
$(A,\cdot,\Delta,\partial,\dh)$ is a \dASIb \wt.
\mlabel{it:rbb3}
\end{enumerate}
\mlabel{cor:3.13}
\end{cor}

\begin{cor}\mlabel{cor:3.14}
	Let $\big((A\bowtie A^*, \star, \frakB_d),(A,\cdot ),(A^{*},\circ )\big)$ be a double construction of Frobenius algebras \wt and $\pa,\dh:A\rightarrow A$ be linear maps.
\begin{enumerate}
\item\mlabel{it:3121}
Then  $\big(A\bowtie A^*, \star, \pa+ \dh^*, \frakB_d\big)$ is a double construction of differential Frobenius algebras \wt if and only if both $(A,\cdot,\pa,\dh)$ and $(A^{*},\circ,\dh^{*},\pa^{*})$ are admissible differential algebras \wt.
\item\mlabel{it:3122}
Under the condition of Item~\meqref{it:3121}, $(A\bowtie A^*,\star, \pa+ \dh^*, \dh+\pa^{*})$ is  an admissible differential algebra \wt containing $(A,\cdot,\pa,\dh)$ and $(A^{*},\circ,\dh^{*},\pa^{*})$ as differential subalgebras.
\end{enumerate}
\end{cor}
\begin{proof}(\mref{it:3121}) follows from Corollary~\mref{cor:3.13}, Lemma~\mref{lem:2.3} and Definition~\mref{de:dasib}.

\smallskip
\noindent
(\mref{it:3122}) follows from Lemma~\mref{lem:3.11} and Proposition~\mref{pp:3.9}.
\end{proof}

\subsection{Noncommutative Novikov bialgebras from differential \asi bialgebras of weight $0$}

\begin{pro}\mlabel{pro:kjl}
\cite{Lo1,SK23} Let $(A,\cdot,\partial)$ be a differential algebra of weight $0$. Then the binary operations
\begin{equation}
x \prec_A y:=x\cdot \partial(y), \quad x\succ_A y:= \partial(x)\cdot y,\quad x,y\in A,\mlabel{eq:11}
\end{equation}
makes $(A,\prec_A, \succ_A)$ into a \nva.
\end{pro}
\begin{pro}\mlabel{pro:5.9}
Let	$(A,\Delta,\dh)$ be a differential coalgebra of weight $0$. Let $\circ:A^*\ot A^*\to A^*$ be the linear dual of $\Delta$.
\begin{enumerate}
\item\mlabel{it:591}
Then $(A^*,\prec_{A^*},\su_{A^*})$ is a \nva, where $\prec_{A^*},\su_{A^*}$ are given by 
\begin{equation}
a^* \prec_{A^*} b^*:= a^*\circ \dh^* (b^*),\quad a^*\succ_{A^*} b^*:= \dh^*(a^*)\circ b^*,\quad a^*,b^*\in A^*.\mlabel{eq:12}
\end{equation}

\item \mlabel{it:592}Let $\Delta_{\prec_{A}}$ and $\Delta_{\su_{A}}$ be the linear duals of $\prec_{A^{*}}$ and $\su_{A^{*}}$, respectively. Then $(A,\Delta_{\prec_{A}},\Delta_{\su_{A}})$ is  a \nvca.

\end{enumerate}
\end{pro}
\begin{proof}
(\mref{it:591}) If $(A,\Delta,\dh)$ is a differential coalgebra of weight $0$, then $(A^*,\circ,\dh^*)$ is a differential algebra of weight $0$. By Proposition~\mref{pro:kjl}, $(A^*,\prec_{A^*},\su_{A^*})$ is a \nva.
\smallskip

\noindent
(\mref{it:592}) This follows from Item~\meqref{it:591} and Lemma~\mref{lem:nva-equi}.
\end{proof}

\begin{pro}\mlabel{pro:ll}
	Let $(A\bowtie A^{*},\star,\partial+\dh^*,\frak B_{d})$ be a double construction of differential Frobenius algebras of weight $0$ associated to (A, $\cdot$, $\partial$) and $(A^*,\circ,\dh^*)$ given by Definition \mref{de:3.10}. Then $\dh(x)\cdot y=-\partial (x)\cdot y$ and $\partial^*(a^*)\circ b^*=-\dh^*(a^*)\circ b^*$ if and only if there is a Manin triple of noncommutative Novikov algebra $((A\bowtie A^{*},\prec_{A\bowtie A^{*}},\su_{A\bowtie A^{*}},\frak B_{d}), A, A^*)$, where binary operations $\prec_{A\bowtie A^{*}},\su_{A\bowtie A^{*}}$ are defined by
	\begin{align}
			(x+a^*)\prec_{A\bowtie A^{*}}(y+b^*)&=(x+a^*)\star (\partial+\dh^*)(y+b^*),\mlabel{eq:xb1}\\
	(x+a^*)\su_{A\bowtie A^{*}}(y+b^*)&=(\partial+\dh^*)(x+a^*)\star (y+b^*), \quad x,y \in A, a^*,b^*\in A^*.\mlabel{eq:xb2}		
\end{align}
\end{pro}
	\begin{proof}
($\Rightarrow$) Firstly, since $(A\bowtie A^{*},\star,\partial+\dh^*)$ is a differential algebra of weight $0$, by Proposition \mref{pro:kjl}, $(A\bowtie A^{*},\prec_{A\bowtie A^{*}},\su_{A\bowtie A^{*}})$ is a noncommutative Novikov algebra. 
Secondly, by Definition \mref{de:3.10}, $\frak B_{d}$ is invariant on $(A\bowtie A^*,\star)$. Then we get
	\begin{align*}
		\frakB_d((x+a^*)\prec_{A\bowtie A^{*}}(y+b^*),z+c^*)&=\frakB_d((x+a^*)\star(\partial+\dh^*)(y+b^*),z+c^*)\\&=\frakB_d(x+a^*,(\partial+\dh^*)(y+b^*)\star (z+c^*))\\&=	\frakB_d((x+a^*),(y+b^*)\su_{A\bowtie A^{*}}(z+c^*)).
		\end{align*}
 By Lemma \mref{lem:3.11}, $\widehat{\partial+\dh^*}$ = $\dh+\partial^*$. By Corollary~\mref{cor:3.13},  $(A,\cdot,\Delta,\partial,\dh)$ is a  \dASIb of weight $0$. Then $(A,\cdot,\partial,\dh)$ is an admissible differential algebra and $(A,\Delta,\dh,\partial)$ is an admissible differential coalgebra. So by Eqs. ~\meqref{eq:pduqr}--\meqref{eq:pduql} and $\dh(x)\cdot y=-\pa(x)\cdot y$, we have $x\cdot \dh(y)=-x\cdot\pa(y)$, and thus $\dh(x\cdot y)=-\partial (x\cdot y)$.
By Eqs.~\meqref{eq:pduqrd}--\meqref{eq:pduqld} and $\partial^*(a^*)\circ b^*=-\dh^*(a^*)\circ b^*$, we get $a^*\circ \pa^*(b^*)=-a^*\circ \dh^*(b^*)$, and hence $\partial^*(a^*\circ b^*)=-\dh^*(a^*\circ b^*)$.
Then we get
\begin{eqnarray*}
\langle\dh(L^*_\circ(b^*)x+R^*_\circ(a^*) y),c^*\rangle&=& \langle x, b^*\circ \dh^*(c^*)\rangle+\langle y, \dh^*(c^*)\circ a^*\rangle\\
&=&- \langle x, b^*\circ \pa^*(c^*)\rangle-\langle y, \pa^*(c^*)\circ a^*\rangle\\
&=&\langle-\pa(L^*_\circ(b^*)x+R^*_\circ(a^*) y),c^*\rangle.\\
\langle \pa^*(R^*_\cdot(x)b^*+L^*_\cdot(y)a^*), z\rangle&=&\langle \pa^*(R^*_\cdot(x)b^*+L^*_\cdot(y)a^*), z\rangle\\
&=&\langle b^*, \pa(z)\cdot x\rangle+\langle a^*,y\cdot \pa(z)\rangle\\
&=&\langle b^*, -\dh(z)\cdot x\rangle-\langle a^*,y\cdot \dh(z)\rangle\\
&=&\langle -\dh^*(R^*_\cdot(x)b^*+L^*_\cdot(y)a^*), z\rangle.
\end{eqnarray*}
This gives $$\dh(L^*_\circ(b^*)x+R^*_\circ(a^*) y)=-\pa(L^*_\circ(b^*)x+R^*_\circ(a^*) y),\,\pa^*(R^*_\cdot(x)b^*+L^*_\cdot(y)a^*)=- \dh^*(R^*_\cdot(x)b^*+L^*_\cdot(y)a^*).$$
Thus, we obtain
\begin{eqnarray*}
(\dh+\partial^*)((x+a^*)\star(y+b^*))&=&(\dh+\pa^*)(x\cdot y+L^*_\circ (b^*) x+R^*_\circ(a^*) y+a^*\circ b^*+R^*_\cdot(x)b^*+L^*_\cdot(y)a^*)\\
&=&\dh(x\cdot y+L^*_\circ (b^*) x+R^*_\circ(a^*))+\pa^*(a^*\circ b^*+R^*_\cdot(x)b^*+L^*_\cdot(y)a^*)\\
&=&-\pa((x\cdot y+L^*_\circ (b^*) x+R^*_\circ(a^*))-\dh^*(a^*\circ b^*+R^*_\cdot(x)b^*+L^*_\cdot(y)a^*)\\
&=&-(\pa+\dh^*)((x+a^*)\star(y+b^*)).
\end{eqnarray*}
Hence, we have
	\begin{eqnarray*}
		&&-\frakB_d(y+b^*,(z+c^*)\ob_{A\bowtie A^{*}}(x+a^*))\\
&=&-\frakB_d(y+b^*,(z+c^*)\star (\partial+\dh^*) (x+a^*)+(\partial+\dh^*)(z+c^*)\star(x+a^*))\\&=&-\frakB_d(z+c^*,(\partial+\dh^*) (x+a^*)\star(y+b^*) )-\frakB_d(z+c^*,\widehat{(\partial+\dh^*)}((x+a^*)\star(y+b^*)))\\
&=&	-\frakB_d(z+c^*,(\partial+\dh^*) (x+a^*)\star(y+b^*)+(\dh+\partial^*)((x+a^*)\star(y+b^*)))\\&=&-\frakB_d(z+c^*,(\partial+\dh^*) (x+a^*)\star(y+b^*)-(\partial+\dh^*)((x+a^*)\star(y+b^*)))\\&=&-\frakB_d(z+c^*,-(x+a^*)\star(\partial+\dh^*) (y+b^*))\\&=&\frakB_d((x+a^*)\prec_{A\bowtie A^{*}}(y+b^*),z+c^*).
		\end{eqnarray*}
So we obtain 
\begin{align*}\frakB_d((x+a^*)\prec_{A\bowtie A^{*}}(y+b^*),z+c^*)&=\frakB_d((x+a^*),(y+b^*)\su_{A\bowtie A^{*}}(z+c^*))\\
&=-\frakB_d(y+b^*,(z+c^*)\ob_{A\bowtie A^{*}}(x+a^*)),
\end{align*}
proving $\frakB_d$ is invariant on $(A\bowtie A^*, \prec_{A\bowtie A^{*}},\su_{A\bowtie A^{*}})$. 
 Thus, $((A\bowtie A^{*},\prec_{A\bowtie A^{*}},\su_{A\bowtie A^{*}},\frak B_{d}), A, A^*$) is a Manin triple of noncommutative Novikov algebras.
\smallskip

\noindent
($\Leftarrow$) By $\frakB_d((x+a^*)\prec_{A\bowtie A^{*}}(y+b^*),z+c^*)=-\frakB_d(y+b^*,(z+c^*)\ob_{A\bowtie A^{*}}(x+a^*))$ and
the above proof, we obtain
$$-(\partial+\dh^*) (x+a^*)\star(y+b^*)-(\dh+\partial^*)((x+a^*)\star(y+b^*))=(x+a^*)\star(\partial+\dh^*) (y+b^*),$$
which leads to $-\pa(x)\cdot y-\dh(x\cdot y)=x\cdot \pa(y)$ and $-\dh^*(a^*)\circ b^*-\pa^*(a^*\circ b^*)=a^*\circ \dh^*(b^*)$.
So by Eqs. ~\meqref{eq:pduqr} and~\meqref{eq:pduqrd}, we get $-\pa(x)\cdot y=\dh(x)\cdot y$ and $-\dh^*(a^*)\circ b^*=\pa^*(a^*)\circ b^*$.
	\end{proof}

With the preparatory work above, we now give the main result of this section.
\begin{thm}\mlabel{thm:main}
Let $(A,\cdot,\Delta, \partial, \dh)$ be a differential \asi bialgebra of weight $0$. Let $\prec_A,\succ_A$  and $\prec_{A^{*}},\su_{A^{*}}$ be given by  Eqs.~\meqref{eq:12} and \meqref{eq:11}, respectively. Let
$\Delta_{\prec_{A}},\Delta_{\su_{A}}$ be the linear duals of $\prec_{A^{*}},\su_{A^{*}}$. If $\dh(x)\cdot y=-\partial (x)\cdot y$
 and $(\dh \ot \id )\Delta =-(\partial \ot \id)\Delta$, then $(A,\prec_A,\succ_A,\Delta_{\prec_{A}},\Delta_{\su_{A}})$ is a \nvbia. In particular, if $\dh=-\pa$, then $(A,\prec_A,\succ_A,\Delta_{\prec_{A}},\Delta_{\su_{A}})$ is  a \nvbia.
\end{thm}
\begin{proof}
	By Corollary~\mref{cor:3.13}, $\big(A\bowtie A^*, \star, \pa+ \dh^*, \frakB_d\big)$ is a double construction of differential Frobenius algebras of weight $0$. Let $\circ$ be the linear dual of $\Delta$. Then  $(\dh \ot \id )\Delta =-(\partial \ot \id)\Delta$ is equivalent to  $\partial^*(a^*)\circ b^*=-\dh^*(a^*)\circ b^*$. Then by Proposition \mref{pro:ll}, there is a Manin triple of noncommutative Novikov algebras  $\big( (A\bowtie A^{*},\prec_{A\bowtie A^{*}},\su_{A\bowtie A^{*}},\frak B_{d})$, $(A,\prec_{A},\su_{A}), (A^*,\prec_{A^{*}},\su_{A^{*}})\big)$, where $\prec_{A\bowtie A^{*}}, \su_{A\bowtie A^{*}}$ are defined by Eqs.~\meqref{eq:xb1}--\meqref{eq:xb2}. Moreover, by Theorem \mref{thm:mpmt}, $(A,\prec_A,\succ_A,\Delta_{\prec_{A}},\Delta_{\su_{A}})$ is a \nvbia.
	\end{proof}

\noindent {\bf Acknowledgments}: This work is supported by the NNSFC (12461002,12326324), the Jiangxi Provincial Natural Science Foundation (20224BAB201003),
the Innovation Fund Designated for Graduate Students of Jiangxi Province(YC2024-S238).
The first author thanks the Chern Institute of Mathematics at Nankai University for hospitality.

\noindent
{\bf Declaration of interests. } The authors have no conflicts of interest to disclose.

\noindent
{\bf Data availability. } Data sharing is not applicable as no data were created or analyzed.

\end{document}